\documentclass[11pt]{amsart} 

\usepackage{amssymb,cite,enumitem,verbatim}
\usepackage[margin=1in]{geometry}
\usepackage{tensor}
\usepackage{graphicx}
\usepackage[all,cmtip]{xy}

\expandafter\let\csname ver@amsthm.sty\endcsname\relax
\let\theoremstyle\relax

\usepackage{hyperref} 
\usepackage{amsthm} 
\usepackage[capitalise]{cleveref} 

\newcommand{\relphantom}[1]{\mathrel{\phantom{#1}}}
\newcommand{\subgrp}[1]{\langle #1 \rangle}
\newcommand{\set}[1]{\left\{ #1 \right\}}
\newcommand{\abs}[1]{\left| #1 \right|}
\newcommand{\bs}[1]{\boldsymbol{#1}}
\newcommand{\wt}[1]{\widetilde{ #1}}
\newcommand{\ol}[1]{\overline{#1}}

\DeclareMathOperator{\cx}{cx}

\DeclareMathOperator{\Ext}{Ext}

\DeclareMathOperator{\opH}{H}
\newcommand{\Hbul}{\opH^\bullet}
\DeclareMathOperator{\Hom}{Hom}
\DeclareMathOperator{\id}{id}
\DeclareMathOperator{\im}{im}
\DeclareMathOperator{\Lie}{Lie}
\DeclareMathOperator{\Max}{Max}
\DeclareMathOperator{\Nil}{Nil}

\DeclareMathOperator{\sdim}{sdim}
\DeclareMathOperator{\sgn}{sgn}
\DeclareMathOperator{\Spec}{Spec}
\DeclareMathOperator{\str}{str}

\newcommand{\simrightarrow}{\stackrel{\sim}{\rightarrow}}

\renewcommand{\mod}{\, \textup{mod}\, }

\newcommand{\gotimes}{\tensor[^g]{\otimes}{}}

\newcommand{\ev}{\textup{ev}}
\newcommand{\odd}{\textup{odd}}
\newcommand{\ve}{\varepsilon}

\newcommand{\N}{\mathbb{N}}
\newcommand{\Z}{\mathbb{Z}}

\newcommand{\Chi}{\mathcal{X}}
\newcommand{\cp}{\mathcal{P}}

\newcommand{\csalg}{\mathfrak{csalg}}
\newcommand{\fa}{\mathfrak{a}}
\newcommand{\fS}{\mathfrak{S}}
\newcommand{\fsmod}{\mathfrak{smod}}
\newcommand{\fsvec}{\mathfrak{svec}}
\newcommand{\g}{\mathfrak{g}}
\newcommand{\gl}{\mathfrak{gl}}
\newcommand{\grp}{\mathfrak{grp}}

\newcommand{\fs}{\mathfrak{s}}

\newcommand{\zero}{\ol{0}}
\newcommand{\one}{\ol{1}}

\newcommand{\Aone}{A_{\one}}
\newcommand{\Azero}{A_{\zero}}
\newcommand{\gone}{\g_{\one}}
\newcommand{\gzero}{\g_{\zero}}
\newcommand{\Vone}{V_{\one}}
\newcommand{\Vzero}{V_{\zero}}

\newcommand{\bsa}{\bs{A}}
\newcommand{\bsc}{\bs{c}}
\newcommand{\bse}{\bs{e}}
\newcommand{\bsi}{\bs{I}}
\newcommand{\bsl}{\bs{\Lambda}}
\newcommand{\bsv}{\bs{\mathcal{V}}}

\newcommand{\bsir}{\bsi^{(r)}}
\newcommand{\bsirone}{{\bsi_1}^{(r)}}
\newcommand{\bsirzero}{{\bsi_0}^{(r)}}
\newcommand{\bsiizero}{{\bsi_0}^{(i)}}
\newcommand{\bsp}{\bs{\cp}}

\newcommand{\glmm}{\gl(m|m)}
\newcommand{\glmn}{\gl(m|n)}
\newcommand{\glone}{\gl(m|n)_{\one}}
\newcommand{\glzero}{\gl(m|n)_{\zero}}

\numberwithin{equation}{subsection}

\newtheorem{theorem}{Theorem}[subsection]
\newtheorem{proposition}[theorem]{Proposition}

\newtheorem{corollary}[theorem]{Corollary}
\newtheorem{lemma}[theorem]{Lemma}

\theoremstyle{definition}
\newtheorem{definition}[theorem]{Definition}

\newtheorem{example}[theorem]{Example}
\newtheorem{remark}[theorem]{Remark}

\crefname{theorem}{Theorem}{Theorems}
\crefname{proposition}{Proposition}{Propositions}
\crefname{conjecture}{Conjecture}{Conjectures}
\crefname{assumption}{Assumption}{Assumptions}
\crefname{corollary}{Corollary}{Corollaries}
\crefname{lemma}{Lemma}{Lemmas}
\crefname{problem}{Problem}{Problems}
\crefname{definition}{Definition}{Definitions}
\crefname{question}{Question}{Questions}
\crefname{example}{Example}{Examples}
\crefname{remark}{Remark}{Remarks}

\title{On support varieties for Lie superalgebras and finite supergroup schemes}

\author{Christopher M.\ Drupieski}
\address{Department of Mathematical Sciences,
		DePaul University,
		Chicago, IL 60614, USA}
\email{cdrupies@depaul.edu}

\author{Jonathan R. Kujawa}
\address{Department of Mathematics \\
		University of Oklahoma \\
		Norman, OK 73019, USA}
\email{kujawa@math.ou.edu}

\thanks{The first author was supported in part by a Faculty Summer Research Grant from the DePaul University College of Science and Health. The second author was supported in part by NSF grant DMS-1160763.}


\allowdisplaybreaks

\begin{document}

\begin{abstract}
We study the spectrum of the cohomology rings of cocommutative Hopf super\-algebras, restricted and non-restricted Lie super\-algebras, and finite supergroup schemes. We also investigate support varieties in these settings and demonstrate that they have the desirable properties of such a theory. We completely characterize support varieties for finite supergroup schemes over algebraically closed fields of characteristic zero, while for non-restricted Lie superalgebras we obtain results in positive characteristic that are strikingly similar to results of Duflo and Serganova in characteristic zero. Our computations for restricted Lie superalgebras and infinitesimal supergroup schemes provide natural generalizations of foundational results of Friedlander and Parshall and of Bendel, Friedlander, and Suslin in the classical setting.
\end{abstract}

\maketitle

\section{Introduction}

\subsection{Overview}

For more than three decades geometric techniques have played a fundamental role in the study of non-semisimple representation theory. Inspired by their use in the study of finite groups \cite{Quillen:1971,Carlson:1983}, Friedlander and Parshall \cite{Friedlander:1986b,Friedlander:1987} introduced support varieties for restricted Lie algebras. Their results were later generalized by Bendel, Friedlander, and Suslin \cite{Suslin:1997,Suslin:1997a} to infinitesimal group schemes and then by Friedlander and Pevtsova \cite{Friedlander:2007} to arbitrary finite group schemes. In a different direction, Ginzburg and Kumar \cite{Ginzburg:1993} calculated the cohomological spectrum for Lusztig's small quantum group. Since then support varieties have been studied for many other interesting classes of finite-dimensional algebras; cf.\ \cite{Aramova:2000, Feldvoss:2011, Benson:2007, Bergh:2009, Pevtsova:2009, Erdmann:2004, Nakano:1998} and the references therein. In each context support varieties have provided important new insights. For example, they play an important role in the study of representation type of self-injective algebras \cite{Feldvoss:2011,Farnsteiner:2007}, in the classification of thick tensor ideals in triangulated categories \cite{Benson:2011, Benson:2015}, in Premet's proof of the Kac--Weisfeiler conjecture on the $p$-divisibility of the dimensions of modules for Lie algebras \cite{Premet:1997}, and in the derived equivalences of Arkhipov, Bezrukavnikov, and Ginzburg \cite{Arkhipov:2004} that relate representations of quantum groups at a root of unity to the geometry of the nilpotent cone.

In contrast to ordinary Lie algebras, the category of finite-dimensional modules of a simple Lie superalgebra over the complex numbers is almost always a non-semisimple category. Supergroups, Lie superalgebras, and related $\Z_2$-graded structures (including $\Z$-graded Hopf algebras as defined by Milnor and Moore \cite{Milnor:1965}) thus provide another natural setting for geometric methods. With Boe and Nakano, the second author initiated a study of support varieties for complex Lie superalgebras and showed that they capture information about the representation theory of these algebras, including atypicality, complexity, and the thick tensor ideals of the category \cite{Boe:2009, Boe:2010, Boe:2012, Boe:2014}. In independent work, Duflo and Serganova \cite{Duflo:2005} also defined associated varieties for Lie superalgebras in characteristic zero and showed they too provide representation theoretic information. 

Much less is known about the representation theory of Lie superalgebras and related $\Z_2$-graded objects over fields of positive characteristic. In a series of papers\cite{Drupieski:2013,Drupieski:2016}, the first author proved that the cohomology ring of a finite supergroup scheme is always a finitely generated algebra, and hence showed that one can define support varieties in this setting. Nevertheless, the study of support varieties for finite supergroup schemes is in its infancy. The results of this paper are a first step toward developing this theory.

\subsection{Main results}

The ambient geometry for support varieties is typically provided by the spectrum of the cohomology ring of the relevant category. Having an explicit description of the spectrum is key for both concrete calculations and for theoretical results. 
For example, if $\g$ is a restricted Lie algebra over an algebraically closed field of characteristic $p$, then the spectrum of the cohomology ring of the restricted enveloping algebra $V(\g)$ is homeomorphic to the restricted nullcone:
\[
\mathcal{N}_1(\g) = \set{x \in \g : x^{[p]}=0 }.
\]
Another similar result is that the spectrum of the cohomology ring for the small quantum group is homeomorphic (under mild assumptions on the root of unity) to the nilpotent cone \cite{Ginzburg:1993}. The goal of the present work is to obtain analogous results in the $\Z_2$-graded setting. Through intricate calculations and through the application of a variety of classical and modern arguments, we obtain a description of the spectrum of the cohomology ring and study support varieties in several natural settings. We briefly describe these results below.

After we develop the necessary preliminaries in \cref{S:prelims}, we consider finite supergroup schemes over an algebraically closed field $k$ of characteristic zero.  A foundational result of Kostant \cite{Kostant:1977} implies that every finite-dimensional cocommutative Hopf superalgebra $A$ over $k$ is the smash product of an exterior algebra and the group algebra of a finite group: $A=\Lambda(V) \# kG$.  We compute the cohomology ring of $A$ in \cref{cohomologychar0} and then show as a consequence that the spectrum of $\Hbul(A,k)$ is the quotient variety $V/G$.  We further prove that the accompanying support variety theory has all the desirable properties of such a theory, including a rank variety description and the tensor product property. As a corollary we obtain a two divisibility result that can be viewed as an analogue of the Kac--Weisfieler conjecture in this setting (cf.\ \cite{Boe:2009} for a similar two divisibility result for complex Lie superalgebras).

In \cref{sec:fdLSAs} we investigate the cohomological spectrum of the enveloping algebra of a finite-dimensional Lie superalgebra $\g = \gzero \oplus \gone$ over an algebraically closed field of odd characteristic. In contrast to ordinary Lie algebras, the cohomology ring of a Lie superalgebra in odd characteristic can provide a nontrivial ambient geometry. Specifically, in \cref{nonrestrictedspectrum} we show that the spectrum of the cohomology ring $\Hbul(\g,k)$ is homeomorphic to 
\[
\Chi_{\g}(k) = \set{ x \in \gone : [x,x]=0 },
\]
where $[\cdot,\cdot]$ denotes the Lie bracket of $\g$. In \cref{associatedvariety} we show that if $M$ is a finite-dimensional $\g$-supermodule, then the support variety for $M$ admits the following rank variety description:
\begin{equation*}
\Chi_{\g}(M) = \set{x \in \Chi_{\g}(k) : M \text{ is not free as a $\subgrp{x}$-supermodule} } \cup \set{0}.
\end{equation*}
It is remarkable that while these support varieties are defined using cohomology in positive characteristic, their rank variety incarnation is identical in definition to the associated varieties defined by Duflo and Serganova \cite{Duflo:2005} for Lie superalgebras in characteristic zero. As far as we know there is no known cohomological definition for Duflo and Serganova's associated varieties, but these results suggest it may be worthwhile to reconsider this question. Conversely, it would be of interest to obtain analogues of the results of Duflo and Serganova in positive characteristic. The calculations in this setting are also remarkable for their similarity to the classical results for ordinary restricted Lie algebras in characteristic $2$ \cite{Jantzen:1986}.

In \cref{S:restrictedLSA} we turn to the cohomological spectrum of a finite supergroup schemes over an algebraically closed field of odd characteristic.  We obtain the strongest results when considering the first Frobenius kernel of the general linear supergroup $GL(m|n)$, or equivalently, when considering its restricted Lie superalgebra $\glmn$. Writing $G$ for the first Frobenius kernel of $GL(m|n)$, we observe from the results of \cite{Drupieski:2016} that there is a finite morphism of varieties, 
\[
\Phi : \Max \left(\Hbul(G,k) \right) \to \mathfrak{gl}(m|n).
\]
In \cref{T:imageofPhi} we prove that the image of $\Phi$ is precisely 
\[
C_{1}(GL(m|n))= \set{ (\alpha, \beta) \in \glzero \times \glone : [\alpha, \beta]=0 \text{ and } \alpha^{[p]}= \tfrac{1}{2}[\beta, \beta] }.
\]
This result holds more generally whenever $G$ is a sub-supergroup scheme of the first Frobenius kernel of $GL(m|n)$. If $n=0$, then the image of $\Phi$ is just the restricted nullcone of the ordinary Lie algebra $\gzero$ as discussed above. The calculations in this section demonstrate an interesting intertwining between the $p$-restricted structure on $\gzero$ and the `$2$-restricted' behavior seen in $\Chi_\g(M)$.

More generally, Bendel, Friedlander, and Suslin \cite{Suslin:1997a, Suslin:1997} proved that if $G_r$ is the $r$-th Frobenius kernel of an affine group scheme $G$, and if $\g = \Lie(G)$ is the Lie algebra of $G$, then (assuming an appropriate embedding of $G$ into some general linear group) the spectrum of the cohomology ring $\Hbul (G_r,k)$ identifies with 
\[
C_r(G)= \set{ (\alpha_0, \ldots , \alpha_{r-1}) \in \mathcal{N}_{1}(\g) : [\alpha_i,\alpha_j] = 0 \text{ for all $i,j$} }.
\]
We show that this characterization is likely to generalize to supergroups. Namely, if we let $G_r$ be the $r$-th Frobenius kernel of $G=GL(m|n)$ and if we set $\g=\glmn$, then one has a finite morphism of varieties,
\[
\Phi_{r} : \Max \left(\Hbul(G_r,k) \right) \to \gzero^{\times r} \times \gone,
\]
and under a certain technical assumption, the image of $\Phi_{r}$ lies in 
\begin{align*}
C_r(G) = \Big\{ (\alpha_0,\alpha_1,\ldots,\alpha_{r-1},\beta) \in (\gzero)^{\times r} \times \gone &: [\alpha_i,\alpha_j] = 0, [\alpha_i,\beta] = 0 \text{ for all $i,j$,} \\
&\alpha_i^{[p]} = 0 \text{ for $0 \leq i \leq r -2$, and } \alpha_{r-1}^{[p]} = \tfrac{1}{2}[\beta,\beta] \Big\}.
\end{align*}

In light of these results and their classical analogues, it is natural to conjecture that $C_r(G)$ describes the cohomology spectrum for the Frobenius kernels of all affine supergroup schemes. More generally, it is an interesting question to generalize the theory of support varieties for infinitesimal and finite group schemes to arbitrary infinitesimal and finite supergroup schemes. For example, in the spirit of Premet's result mentioned above, such a theory would allow one to provide a geometric proof of the super Kac--Weisfeiler Conjecture \cite{Wang:2009}. We expect, however, that the work of producing these generalizations will be quite nontrivial. The examples in \cref{subsec:examples} already demonstrate that even in small examples the existence of odd elements causes new phenomena. 

\subsection{Conventions}

Throughout the paper, $k$ will denote a field of characteristic $p \neq 2$. Beginning in Section \ref{section:char0} we will assume that $k$ is algebraically closed. All vector spaces will be $k$-vector spaces and all unadorned tensor products will denote tensor products over $k$. If $V$ is a $k$-vector space, then $V^*$ will denote its linear dual, i.e., $V^* = \Hom_k(V,k)$. When $p > 0$, let $V^{(1)}$ denote the $k$-vector space obtained by twisting the $k$-module structure on $V$ by the Frobenius map $\lambda \mapsto \lambda^p$.

We will generally follow the notation, terminology, and conventions laid out in \cite[\S2]{Drupieski:2013}. In particular, we will assume that the reader is familiar with the sign conventions of ``super'' linear algebra. We will also assume familiarity with Lie superalgebras---in particular, with the general linear Lie superalgebra $\glmn$---and with the general linear supergroup scheme $GL(m|n)$ (cf.\ \cite{Brundan:2003a}). Set $\Z_2 = \Z/2\Z = \{ \zero,\one \}$, and write $V = \Vzero \oplus \Vone$ for the decomposition of a superspace $V$ into its even and odd subspaces. Recall that the superdimension of a superspace $V$, denoted $\sdim(V)$, is defined by $\sdim(V) = \dim(\Vzero) - \dim(\Vone)$. Given a homogeneous element $v \in V$, write $\ol{v} \in \Z_2$ for the $\Z_2$-degree of $v$. Isomorphisms arising from even linear maps will be denoted by the symbol ``$\cong$'' while isomorphisms arising from odd linear maps will be denoted by the symbol ``$\simeq$''. Write $\N$ for the set $\set{0,1,2,\ldots}$ of non-negative integers.

\subsection{Acknowledgements} The second author is pleased to acknowledge the hospitality and support of the Mittag-Leffler Institute during the special semester in Representation Theory during Spring 2015. 

\section{Preliminaries}\label{S:prelims}

In this section assume that $k$ is a field of characteristic $p \neq 2$.

\subsection{Affine supergroup schemes}

In this section we recall some basic definitions and results regarding affine supergroup schemes. For more details on affine group schemes and affine supergroup schemes, we refer to the reader to \cite{Jantzen:2003} and \cite[\S4]{Drupieski:2013}.

Write $\csalg = \csalg_k$ for the category whose objects are the commutative $k$-superalgebras and whose morphisms are the even superalgebra homomorphisms between them. Then an affine $k$-super\-group scheme is a representable functor from $\csalg$ to the category $\grp$ of groups. In other words, an affine $k$-supergroup scheme $G$ is a functor $G: \csalg \rightarrow \grp$ for which there exists a commutative superalgebra $k[G] \in \csalg$, called the coordinate superalgebra of $G$, such that for each $A \in \csalg$, $G(A) = \Hom_{\csalg}(k[G],A)$. As for ordinary affine group schemes, the group structure maps on $G$ correspond uniquely to comorphisms on $k[G]$; these endow $k[G]$ with the structure of a Hopf superalgebra. Then the category of affine $k$-supergroup schemes is anti-equivalent to the category of commutative $k$-Hopf superalgebras.

\begin{remark}
A $\Z$-graded Hopf algebra in the sense of Milnor and Moore \cite{Milnor:1965} is an example of a Hopf superalgebra. The $\Z_2$-grading on such a Hopf algebra is obtained by simply reducing the $\Z$-grading modulo $2$.
\end{remark}

An affine $k$-supergroup scheme $G$ is \emph{algebraic} if $k[G]$ is a finitely-generated $k$-superalgebra, and is \emph{finite} if $k[G]$ is a finite-dimensional $k$-algebra. If $G$ is a finite $k$-supergroup scheme, then the Hopf superalgebra structure maps on $k[G]$ induce by duality the structure of a cocommutative Hopf superalgebra on $k[G]^*$. We denote $k[G]^*$ by $kG$, and call $kG$ the group algebra of $G$. The category of finite $k$-supergroup schemes is thus equivalent to the category of finite-dimensional cocommutative $k$-Hopf superalgebras. Given a finite $k$-supergroup scheme $G$, the category of rational $G$-supermodules (i.e., the category of $k[G]$-supercomodules) is naturally equivalent to the category of $kG$-supermodules; see \cite[\S4.3]{Drupieski:2013}.

A finite supergroup scheme $G$ is \emph{infinitesimal} if the augmentation ideal $I_\ve$ of $k[G]$ is nilpotent. If $G$ is infinitesimal and if the characteristic $p$ of the field $k$ is not $0$, then the minimal non-negative integer $r$ such that $x^{p^r} = 0$ for all $x \in I_\ve$ is called the \emph{height} of $G$. For example, if $G$ is an arbitrary affine $k$-supergroup scheme, then the $r$-th Frobenius kernel of $G$, which is the sub-supergroup scheme of $G$ defined by the ideal $\set{f^{p^r}: f \in I_\ve}$, is infinitesimal of height $r$. Finite supergroup schemes over algebraically closed fields of characteristic $0$ are classified by \cref{cor:Kostant} below, while over perfect fields of characteristic $p \geq 3$ one has the following theorem \cite[Theorem 5.3.1]{Drupieski:2013}:

\begin{theorem}
Let $k$ be a perfect field of characteristic $p \geq 3$, and let $G$ be an affine algebraic $k$-supergroup scheme. Then there exists an etale (ordinary) group scheme $\pi_0(G)$ and a normal sub-supergroup scheme $G^0$ of $G$ such that $G/G^0 \cong \pi_0(G)$. If $G$ is finite, then $G^0$ is infinitesimal and $G \cong G^0 \rtimes \pi_0(G)$.
\end{theorem}

Let $G$ be an affine $k$-supergroup scheme. The Lie superalgebra of $G$, denoted $\Lie(G)$, is defined in \cite[\S4.2]{Drupieski:2013}. Set $\g = \Lie(G)$. If $p > 0$, then $\g$ is naturally a restricted Lie superalgebra. We denote the restricted enveloping superalgebra of $\g$ by $V(\g)$. If $G$ is a height-one infinitesimal group scheme, then the group algebra $kG$ of $G$ identifies with $V(\g)$ \cite[Lemma 4.4.2]{Drupieski:2013}. As discussed in \cite[Remark 4.4.3]{Drupieski:2013}, the category of height-one infinitesimal $k$-supergroup schemes is naturally equivalent to the category of finite-dimensional restricted Lie superalgebras over $k$.

\subsection{Graded superalgebras}

A \emph{graded superalgebra} is a $\Z \times \Z_2$-graded algebra. A homomorphism of graded superalgebras is an algebra homomorphism that preserves the $\Z \times \Z_2$-gradings. Given a graded superalgebra $A$ and a homogeneous element $a \in A$, we will write $\deg(a)$ for the $\Z$-degree of $a$ and $\ol{a}$ for the $\Z_2$-degree of $a$.

\begin{definition}
A graded superalgebra $A$ is \emph{graded-commutative} if for all homogeneous elements $a,b \in A$, one has
\[
ab = (-1)^{\ol{a} \cdot \ol{b} + \deg(a) \cdot \deg(b)} ba.\footnote{From now on, whenever we state a formula in which homogeneous degrees have been specified, we mean that the formula is true as written for homogeneous elements and that it extends linearly to non-homogeneous elements.}
\]
\end{definition}

If $A$ is a graded superalgebra concentrated in $\Z$-degree $0$, then $A$ is graded-commutative if and only if $A$ is commutative in the usual sense for superalgebras. (In this paper, the term \emph{commutative}, as applied to superalgebras, will always be used in the sense indicated here, while the usual notion of commutativity for abstract rings will be referred to as \emph{ordinary} commutativity.)

\begin{definition}
Let $A$ and $B$ be graded superalgebras. Then the \emph{graded tensor product} of $A$ and $B$, denoted $A \gotimes B$, is the graded superalgebra whose underlying superspace is the tensor product of superspaces $A \otimes B$, in which the $\Z$-degree of homogeneous simple tensors is defined by $\deg(a \otimes b) = \deg(a) + \deg(b)$, and the product of homogeneous simple tensors is defined by
\[
(a \otimes b)(c \otimes d) = (-1)^{\ol{b} \cdot \ol{c} + \deg(b) \cdot \deg(c)} (ac \otimes bd).
\]
\end{definition}

Given a graded superalgebra $A = \bigoplus_{n \in \Z} A^n$, set
\begin{align*}
A^\ev &= \bigoplus_{n \in \Z} A^{2n}, & \Azero^\ev &= (A^\ev)_{\zero}, & \Aone^\ev &= (A^\ev)_{\one}, \\
A^\odd &= \bigoplus_{n \in \Z} A^{2n+1}, & \Azero^\odd &= (A^\odd)_{\zero}, & \Aone^\odd &= (A^\odd)_{\one}.
\end{align*}
If $A$ is graded-commutative, then $\Azero^\ev \oplus \Aone^\odd$ is a commutative subalgebra of $A$ in the ordinary sense, and (since the characteristic of $k$ is not $2$) the elements of $\Azero^\odd$ and $\Aone^\ev$ square to $0$.

\begin{lemma} \label{I+radical}
Let $A$ be a graded-commutative superalgebra. Set $R = \Azero^\ev \oplus \Aone^\odd$, and let $I \unlhd R$ be a homogeneous ideal containing $\Nil(R)$, the nilradical of $R$. Then $I^+ := I \oplus \Azero^\odd \oplus \Aone^\ev$ is a homogeneous ideal in $A$. If $I$ is a radical ideal in $R$, then $I^+$ is a radical ideal in $A$ in the sense
\[
I^+ = \sqrt{I^+} := \set{x \in A : x^n \in I^+ \text{ for some $n \in \N$}}.
\]
\end{lemma}

\begin{proof}
Let $I$ be a homogeneous ideal in $R$. Then $I^+$ is a homogeneous subspace of $A$, so to prove the first claim it suffices to check that $I^+$ is closed under multiplication; we verify this separately for multiplication by elements of $R$ and multiplication by elements of $\Azero^\odd \oplus \Aone^\ev$.

Since $\Azero^\odd \oplus \Aone^\ev$ is closed under multiplication by $R$, the assumption that $I$ is an ideal in $R$ implies that $I^+$ is closed under multiplication by $R$. Next, since $I \subseteq R$, multiplication by $\Azero^\odd \oplus \Aone^\ev$ maps $I$ into the subspace $\Azero^\odd \oplus \Aone^\ev$ of $I^+$. Finally, using the graded-commutativity of $A$ and its consequence that the elements of $\Azero^\odd$ and $\Aone^\ev$ square to zero, one can check that if $x,y \in \Azero^\odd \oplus \Aone^\ev$, then $xy$ is a nilpotent element of $R$. Since $I$ contains the nilradical of $R$, then $xy \in I$. Then $I^+$ is closed under multiplication by $\Azero^\odd \oplus \Aone^\ev$, and hence by all of $A$.

Now suppose that $I$ is a radical ideal in $R$. The inclusion $I^+ \subseteq \sqrt{I^+}$ is tautological, so let $x \in \sqrt{I^+}$ and suppose that $x^n \in I^+$. Then the coset $x+I^+$ is a nilpotent element of the quotient ring $A/I^+$. The inclusion $R \hookrightarrow A$ induces an isomorphism of graded superalgebras $R/I \cong A/I^+$. Since $I$ is a radical ideal in $R$, the quotient ring $R/I$ has no nonzero nilpotent elements. Then neither does $A/I^+$, meaning that $x+I^+$ must be equal to zero in $A/I^+$, and hence that $x \in I^+$.
\end{proof}

From now on, whenever we refer to an ideal in a graded-commutative superalgebra as being a radical ideal, we mean it in the sense of \cref{I+radical}.

\begin{proposition} \label{radicalideal}
Let $A$ be a graded-commutative superalgebra, and let $J \unlhd A$ be a homo\-geneous ideal. Then $\sqrt{J} := \set{ x \in A: x^n \in J \text{ for some $n \in \N$}}$ is a homogeneous radical ideal and $A/\sqrt{J}$ is a reduced commutative ring. Specifically,
\[
\sqrt{J} = \sqrt[R]{J_R} \oplus \Azero^\odd \oplus \Aone^\ev,
\]
where $R = \Azero^\ev \oplus \Aone^\odd$, $J_R = J \cap R$, and $\sqrt[R]{J_R}$ is the radical of $J_R$ in $R$.
\end{proposition}

\begin{proof}
Let $J$ be a homogeneous ideal in $A$. Set $R = \Azero^\ev \oplus \Aone^\odd$, and let $J_R = J \cap R$. Then $J_R$ is a homogeneous ideal in $R$, and $J = J_R \oplus (J \cap \Azero^\odd) \oplus (J \cap \Aone^\ev)$. Let $\sqrt[R]{J_R}$ be the radical of $J_R$ in $R$. By the theory of ordinary commutative graded rings, $\sqrt[R]{J_R}$ is a homogeneous ideal in $R$. Now set $J^+ = \sqrt[R]{J_R} \oplus \Azero^\odd \oplus \Aone^\ev$. Then by \cref{I+radical}, $J^+$ is a homogeneous radical ideal in $A$. Clearly $J \subseteq J^+$, so $\sqrt{J} \subseteq J^+$ because $J^+$ is a radical ideal. Now to prove that $J^+ \subseteq \sqrt{J}$, and hence show that $\sqrt{J}$ is a homogeneous ideal, let $x \in J^+$. Then $x = x_0^0 + x_1^1 + x_0^1 + x_1^0$ for some $x_0^0 \in \Azero^\ev$, $x_1^1 \in \Aone^\odd$, $x_0^1 \in \Azero^\odd$, and $x_1^0 \in \Aone^\ev$. Since $J^+ = \sqrt[R]{J_R} \oplus \Azero^\odd \oplus \Aone^\ev$, and since $\sqrt[R]{J_R}$ is a homogeneous radical ideal in $R$, this implies that $x_0^0,x_1^1 \in \sqrt[R]{J_R}$. Choose $n \in \N$ such that $(x_0^0)^n,(x_1^1)^n \in J_R$. Now expanding the product $(x_0^0+x_1^1+x_0^1+x_1^0)^{2n+2}$, and using the graded-commutativity of $A$, it follows that $x^{2n+2}$ can be written as a linear combination of monomials of the form $(x_0^0)^a (x_1^1)^b (x_0^1)^c (x_1^0)^d$ with $a+b+c+d = 2n+2$. Since $(x_0^1)^2 = (x_1^0)^2 = 0$ by the graded-commutativity of $A$, we may assume that $c$ and $d$ are each at most $1$. Then by the pigeonhole principle, we must have either $a \geq n$ or $b \geq n$. Then each monomial in the expansion of $x^{2n+2}$ has at least one factor in $J_R$. Since $J_R \subseteq J$, and since $J \unlhd A$, this implies that $x^{2n+2} \in J$. Finally, the inclusion $R \hookrightarrow A$ induces an isomorphism of graded superalgebras $R/\sqrt[R]{J_R} \cong A/\sqrt{J}$. Since $R/\sqrt[R]{J_R}$ is a reduced commutative ring in the ordinary sense, so is $A/\sqrt{J}$.
\end{proof}

\begin{corollary} \label{cor:nilradical}
Let $A$ be a graded-commutative graded superalgebra. Set $R = \Azero^\ev \oplus \Aone^\odd$. Then
\[
\Nil(A) := \set{ x \in A: x^n= 0 \text{ for some $n \in \N$}} = \Nil(R) \oplus \Azero^\odd \oplus \Aone^\ev
\]
is a homogeneous ideal in $A$, and $A/\Nil(A)$ is a reduced commutative ring in the ordinary sense.
\end{corollary}

\begin{proof}
Take $J = \set{0}$ in \cref{radicalideal}.
\end{proof}

\begin{definition}
Let $A$ be a graded-commutative super\-algebra. Define $\Spec(A)$ to be the prime ideal spectrum commutative ring $A/\Nil(A)$, considered in the usual manner as a topological space via the Zariski topology. Define $\Max(A)$ be the maximal ideal spectrum of $A/\Nil(A)$, considered as a topological space via the topology inherited from $\Spec(A)$.
\end{definition}

If $A$ is a finitely-generated algebra and if the field $k$ is algebraically closed, then $\Max(A)$ is an affine algebraic variety in the usual sense.

\subsection{Representations, cohomology, and support varieties} \label{subsec:repcohomologysupport}

Let $A$ be a $k$-superalgebra and let $M$ and $N$ be (left) $A$-supermodules. A linear map $g: M \rightarrow N$ is an $A$-supermodule homomorphism if for all $a \in A$ and $m \in M$ one has $g(a.m) = (-1)^{\ol{a} \cdot \ol{g}} a.g(m)$. We denote the vector superspace of all $A$-supermodule homomorphisms from $M$ to $N$ by $\Hom_A(M,N)$. The category $\fsmod_A$, whose objects are the $A$-supermodules and whose morphisms are the $A$-supermodule homomorphisms, is not an abelian category, though the underlying even subcategory $(\fsmod_A)_\ev$, consisting of the same objects but only the even homomorphisms, is an abelian category. When $A$ is clear from the context we may denote $\fsmod_A$ simply by $\fsmod$. Then
\[
\Hom_{\fsmod_\ev}(M,N) = \Hom_A(M,N)_{\zero}.
\]
We write $\fsvec$ for the category of $k$-supermodules, i.e., the category of $k$-superspaces.

Given superspaces $V$ and $W$ and a linear map $\phi: V \rightarrow W$, define $\Pi(V)$ to be the superspace $V$ equipped instead with the opposite $\Z_2$-grading and define $\Pi(\phi): \Pi(V) \rightarrow \Pi(W)$ to be equal to $\phi$ as a linear map between the underlying vector spaces. We call $\Pi: \fsvec \rightarrow \fsvec$ the parity change functor. If $M$ is a left $A$-supermodule, then $\Pi(M)$ is made into a left $A$-supermodule by having each $a \in A$ act on $\Pi(M)$ the way that $(-1)^{\ol{a}} a$ acts on $M$. In other words, if we write $m^\pi$ to denote an element $m \in M$ considered as an element of $\Pi(M)$, then $a.m^\pi = (-1)^{\ol{a}}(a.m)^\pi$. Now the identity map on $M$ defines an odd $A$-supermodule isomorphism $M \simeq \Pi(M)$, $m \mapsto m^\pi$. In particular, if $M$ and $N$ are $A$-supermodules, there are odd isomorphisms
\begin{equation} \label{eq:homparityflip}
\Hom_A(M,N) \simeq \Hom_A(M,\Pi(N)) \quad \text{and} \quad \Hom_A(M,N) \simeq \Hom_A(\Pi(M),N)
\end{equation}
that are natural with respect to even homomorphisms in either variable.

Given a $k$-superalgebra $A$, one can form the smash product algebra $A \# k\Z_2$. As a superspace, $A \# k\Z_2$ is simply $A \otimes k\Z_2$, the tensor product of $A$ with the group ring $k\Z_2$ for $\Z_2$ (considered as a purely even superalgebra). Multiplication in $A \# k\Z_2$ is induced by the given products in $A$ and $k\Z_2$ and by the relation $(1 \otimes \one)(a \otimes \zero) = (-1)^{\ol{a}} a \otimes \one$. In the previous equation, $\zero$ and $\one$ denote the identity and non-identity elements of the group $\Z_2$; these elements form a vector space basis for the group ring $k\Z_2$. The category of $A \# k\Z_2$-modules naturally identifies with $(\fsmod_A)_\ev$. Specifically, if $M$ is an $A \# k\Z_2$-module, then $M$ is an $A$-module and a $k\Z_2$-module by restriction. Since $k$ is assumed to be a field of characteristic $p \neq 2$, $M$ decomposes under the action of $k\Z_2$ into a trivial $k\Z_2$-submodule $M_{\zero}$ and a nontrivial $k\Z_2$-submodule $M_{\one}$. Then $M$ is an $A$-supermodule with respect to the decomposition $M = M_{\zero} \oplus M_{\one}$. Conversely, if $M$ is an $A$-supermodule, then $M$ lifts to the structure of an $A \# k\Z_2$-module by having $\one \in \Z_2$ act on $M$ by the map $m \mapsto (-1)^{\ol{m}}m$.

We say that an $A$-supermodule $P$ is projective if the functor $\Hom_A(P,-): \fsmod_\ev \rightarrow \fsvec_\ev$ is exact, and we say that an $A$-supermodule $Q$ is injective if $\Hom_A(-,Q): \fsmod_\ev \rightarrow \fsvec_\ev$ is exact. It follows from \eqref{eq:homparityflip} that an $A$-supermodule is projective (resp.\ injective) in this sense if and only if it is projective (resp.\ injective) in the usual sense in the abelian category $(\fsmod_A)_\ev$. Since $(\fsmod_A)_\ev$ identifies with the category of $A \# k\Z_2$-modules, it contains enough projectives and enough injectives. We can thus make the following definition:

\begin{definition}
Given an $A$-supermodule $M$, define $\Ext_A^n(M,-)$ as in \cite[\S\S2.5--2.6]{Drupieski:2013} to be the $n$-th right derived functor of $\Hom_A(M,-): (\fsmod_A)_\ev \rightarrow \fsvec_\ev$.
\end{definition}

With the above definition, the extension group $\Ext_A^n(M,N)$ can have a nonzero odd subspace, whereas ordinary extension groups in the abelian category $\fsmod_\ev$, which are defined in terms of the right derived functors of $\Hom_{\fsmod_\ev}(M,-) = \Hom_A(M,-)_{\zero}$, are always purely even superspaces. In principle, the odd isomorphisms \eqref{eq:homparityflip} permit us to reduce all cohomology calculations to calculations strictly within the abelian category $\fsmod_\ev$, but we prefer to work with $\Ext_A^\bullet(M,N)$ so that we don't have to manually keep track of the $\Z_2$-gradings. In terms of the smash product algebra $A \# k\Z_2$, $\Ext_{A \# k\Z_2}^\bullet(M,N) = \Ext_A^\bullet(M,N)_{\zero}$.

If $A$ is a Hopf superalgebra, then $A \# k\Z_2$ is naturally a Hopf algebra in the ordinary sense. Specifically, if the coproduct in $A$ of an element $a \in A$ is denoted $\Delta(a) = \sum a_1 \otimes a_2$, then the coproduct on $A \# k\Z_2$ is determined by the equations
\begin{align*}
\Delta_{A \# k\Z_2}(a \otimes \zero) &= \sum (a_1 \otimes \ol{a_2}) \otimes (a_2 \otimes \zero), \quad \text{and} \\
\Delta_{A \# k\Z_2}(1 \otimes \one) &= (1 \otimes \one) \otimes (1 \otimes \one).
\end{align*}
Since $(\fsmod_A)_\ev$ identifies with the category of $A \# k\Z_2$-modules, the next lemma follows from the well-known fact that any finite-dimensional ordinary Hopf algebra is self-injective.

\begin{lemma} \label{L:selfinjective}
If $A$ is a finite-dimensional Hopf superalgebra, then an $A$-supermodule is projective if and only if it is injective.
\end{lemma}

Recall that elements of $\Ext_{\fsmod_\ev}^n(M,N)$ can be interpreted as equivalence classes of length-$n$ exact sequences, i.e., exact sequence in $\fsmod_\ev$ of the form
\begin{equation} \label{eq:nexactsequence}
E: \quad 0 \rightarrow N \rightarrow E_n \rightarrow E_{n-1} \rightarrow \cdots \rightarrow E_1 \rightarrow M \rightarrow 0.
\end{equation}
Similarly, homogeneous elements of $\Ext_A^n(M,N)$ can be interpreted as equivalence classes of exact sequences of the form \eqref{eq:nexactsequence} in which each arrow is a homogeneous $A$-supermodule homomorphism. The parity of such an equivalence class is then the sum of the parities of the arrows appearing in any representative for it. For details the reader can consult \cite[\S3.5]{Drupieski:2016}, replacing the category $\bsp$ with $\fsmod_A$. Homogeneous exact sequences can be spliced together in the usual way, and this gives rise to an even linear map
\[
\Ext_A^m(M,N) \otimes \Ext_A^n(P,M) \rightarrow \Ext_A^{m+n}(P,N), \quad \alpha \otimes \beta \mapsto \alpha \circ \beta,
\]
that we call \emph{Yoneda composition} or the \emph{Yoneda product}. In particular, $\Hbul(A,k) := \Ext_A^\bullet(k,k)$ is a graded superalgebra. For more discussion of the Yoneda product, including an interpretation via homotopy classes of homogeneous chain maps, see \cite[\S3.2]{Drupieski:2016} (again replacing $\bsp$ with $\fsmod_A$).

Now suppose $A$ is a Hopf superalgebra with coproduct $\Delta: A \rightarrow A \otimes A$ and antipode $S: A \rightarrow A$. Given $a \in A$, write $\Delta(a) = \sum a_1 \otimes a_2$. Let $M$ and $N$ be left $A$-supermodules. Then the action of $A$ on $M \otimes N$ is defined for $m \in M$ and $n \in N$ by $a.(m \otimes n) = \sum (-1)^{\ol{a_2} \cdot \ol{m}} (a_1.m) \otimes (a_2.n)$, and the action of $A$ on $\Hom_k(M,N)$ is defined for $g \in \Hom_k(M,N)$ by $(a.g)(m) = \sum (-1)^{\ol{g} \cdot \ol{a_2}} a_1.g(S(a_2).m)$. If $M$ is finite-dimensional, then the superspace isomorphism $N \otimes M^* \cong \Hom_k(M,N)$ defined for $n \in N$ and $g \in M^*$ by $(n \otimes g) \mapsto (m \mapsto g(m) \cdot n)$ is an $A$-supermodule isomorphism. If $A$ is cocommutative (in the super sense), then the supertwist map $T: M \otimes N \rightarrow N \otimes M$ defined by $m \otimes n \mapsto (-1)^{\ol{m} \cdot \ol{n}} n \otimes n$ is an $A$-supermodule homomorphism.

\begin{lemma} \label{tensorproductdual}
Let $A$ be a Hopf superalgebra, and let $M$ and $N$ be finite-dimensional $A$-super\-modules. Then $(M \otimes N)^* \cong N^* \otimes M^*$ as $A$-supermodules.
\end{lemma}

\begin{proof}
Let $\phi \in N^*$ and let $\psi \in M^*$. Then we consider $\phi \otimes \psi \in N^* \otimes M^*$ as an element of $(M \otimes N)^*$ by setting $(\phi \otimes \psi)(m \otimes n) = \psi(m) \phi(n)$. Since $M$ and $N$ are assumed to be finite-dimensional, this defines a superspace isomorphism $N^* \otimes M^* \cong (M \otimes N)^*$, which the reader can verify is compatible with the action of $A$. (In the course of the verification, it is useful to observe that $\psi(m) = 0$ unless $\ol{\psi} = \ol{m}$, and similarly that $\phi(n) = 0$ unless $\ol{\phi} = \ol{n}$.)
\end{proof}

\begin{lemma} \label{tensorhomadjunction}
Let $A$ be a Hopf superalgebra, let $M$, $N$, and $P$ be $A$-supermodules, and suppose that $N$ is finite-dimensional. Then there exist natural even isomorphisms
\[
\Hom_A(M \otimes N,P) \cong \Hom_A(M,P \otimes N^*) \cong \Hom_A(M,\Hom_k(N,P)).
\]
\end{lemma}

\begin{proof}
Let $\tau: k \rightarrow \Hom_k(N,N) \cong N \otimes N^*$ be the linear map that sends a scalar $\lambda \in k$ to the corresponding scalar multiple of the identity map on $N$, and let $c: N^* \otimes N \rightarrow k$ be the contraction map defined for $g \in N^*$ and $n \in N$ by $c(g \otimes n) = g(n)$. It is straightforward to check that $\tau$ and $c$ are $A$-supermodule homomorphisms. We now define even linear maps
\begin{align*}
\Phi: &\Hom_k(M \otimes N,P) \rightarrow \Hom_k(M,P \otimes N^*), \quad \Phi(f) = (f \otimes N^*) \circ (M \otimes \tau), &\text{and} \\
\Theta: &\Hom_k(M,P \otimes N^*) \rightarrow \Hom_k(M \otimes N,P), \quad \Theta(g) = (P \otimes c) \circ (g \otimes N).
\end{align*}
In other words, $\Phi(f)$ is the composite map
\[
M \stackrel{\sim}{\longrightarrow} M \otimes k \stackrel{M \otimes \tau}{\longrightarrow} M \otimes N \otimes N^* \stackrel{f \otimes N^*}{\longrightarrow} P \otimes N^*,
\]
and $\Theta(g)$ is the composite map
\[
M \otimes N \stackrel{g \otimes N}{\longrightarrow} P \otimes N^* \otimes N \stackrel{P \otimes c}{\longrightarrow} P \otimes c \stackrel{\sim}{\longrightarrow} P.
\]
One can check that $\Phi$ and $\Theta$ are inverse superspace isomorphisms. Since $\tau$ and $c$ are $A$-supermodule homomorphisms, it follows that $\Phi$ and $\Theta$ each carry homomorphisms to homomorphisms.
\end{proof}

One can verify as in \cite[Proposition 3.1.5]{Benson:1998} that if $P$ is a projective $A$-supermodule and if $M$ is an arbitrary $A$-supermodule, then $P \otimes M$ is projective. From this and the K\"{u}nnth formula it follows that a tensor product of projective resolutions in $\fsmod_\ev$ is again a projective resolution. Then as for ordinary Hopf algebras \cite[p.\ 57]{Benson:1998} one can define cup products
\[
\cup: \Ext_A^m(M,N) \otimes \Ext_A^n(M',N') \rightarrow \Ext_A^{m+n}(M \otimes M',N \otimes N'), \quad \alpha \otimes \beta \mapsto \alpha \cup \beta.
\]
Assuming that $\alpha$ and $\beta$ are homogeneous, and representing them by exact sequences $E_\alpha$ and $E_\beta$ of the form \eqref{eq:nexactsequence}, $\alpha \cup \beta$ is represented by the tensor product of complexes $E_\alpha \otimes E_\beta$.

Let $P$ be an arbitrary $A$-super\-module. Since the tensor product functor $- \otimes P: M \mapsto M \otimes P$ is exact, it follows from the interpretation of homogeneous elements in $\Ext_A^\bullet(M,N)$ as equivalence classes of exact sequences that $- \otimes P$ induces an even linear map
\begin{equation} \label{eq:tensorproductonext}
\Phi_P: \Ext_A^\bullet(M,N) \rightarrow \Ext_A^\bullet(M \otimes P,N \otimes P), \quad \alpha \mapsto \alpha \otimes P,
\end{equation}
that is compatible with Yoneda products. In terms of cup products, $\alpha \otimes P = \alpha \cup \id_P$, where $\id_P$ denotes the identity map on $P$. Similarly, taking the left cup product with $\id_P$ defines an even linear map $\Ext_A^\bullet(M,N) \rightarrow \Ext_A^\bullet(P \otimes M, P \otimes N)$, denoted $\alpha \mapsto P \otimes \alpha$.

\begin{proposition} \label{cupcomposition}
Let $\alpha \in \Ext_A^m(M,N)$ and let $\beta \in \Ext_A^n(M',N')$ be homogeneous. Then
\begin{equation} \label{eq:MacLane4.1}
\alpha \cup \beta = (\alpha \otimes N') \circ (M \otimes \beta) = (-1)^{m \cdot n+\ol{\alpha} \cdot \ol{\beta}} [(N \otimes \beta) \circ (\alpha \otimes M')].
\end{equation}
\end{proposition}

\begin{proof}
Let $P_M$, $P_N$, $P_{M'}$, and $P_{N'}$ be projective resolutions in $\fsmod_\ev$ of $M$, $N$, $M'$, and $N'$, respectively. Then by the discussion in \cite[\S3.2]{Drupieski:2016} (replacing $\bsp$ with $\fsmod_A$), the even (resp.\ odd) subspace of $\Ext_A^m(M,N)$ identifies with the vector space of homotopy classes of even (resp.\ odd) chain maps $P_M \rightarrow P_N$ of degree $-m$, and similarly for $\Ext_A^n(M',N')$. We can thus identify $\alpha$ and $\beta$ with representative homogeneous chain maps $\ol{\alpha}: P_M \rightarrow P_N$ and $\ol{\beta}: P_{M'} \rightarrow P_{N'}$ of degrees $-m$ and $-n$, respectively. Now the tensor products of complexes $P_M \otimes P_N$ and $P_{M'} \otimes P_{N'}$ are projective resolutions in $\fsmod_\ev$ for $M \otimes N$ and $M' \otimes N'$, and the cup product $\alpha \cup \beta$ is represented by the chain map $\ol{\alpha} \otimes \ol{\beta}: P_M \otimes P_N \rightarrow P_{M'} \otimes P_{N'}$. On the other hand, the extension class $\alpha \otimes N'$ is represented by the chain map $\ol{\alpha} \otimes P_{N'}$, i.e., the chain map $P_M \otimes P_{N'} \rightarrow P_{M'} \otimes P_{N'}$ that acts via $\ol{\alpha}$ on the first factor and via the identity map on the second factor. Similarly, $M \otimes \beta$, $N \otimes \beta$, and $\alpha \otimes M'$ are represented by $P_M \otimes \ol{\beta}$, $P_N \otimes \ol{\beta}$, and $\ol{\alpha} \otimes P_{M'}$, respectively. With these identifications, the Yoneda product is induced by the composition of chain maps. Now the reader can check \eqref{eq:MacLane4.1} by verifying the relevant equalities between chain map representatives.
\end{proof}

\begin{corollary} \label{cor:gradedcommutativity}
Let $A$ be a Hopf superalgebra and let $M$ be an $A$-supermodule. Then the Yoneda and cup products on $\Hbul(A,k)$ agree, and the image of the algebra homomorphism $\Phi_M: \Hbul(A,k) \rightarrow \Ext_A^\bullet(M,M)$ is central in $\Ext_A^\bullet(M,M)$ in the sense that if $\alpha \in \opH^n(A,k)$ and $\beta \in \Ext_A^m(M,M)$,
\[
\Phi_M(\alpha) \circ \beta = (-1)^{m \cdot n + \ol{\alpha} \cdot \ol{\beta}} \beta \circ \Phi_M(\alpha).
\]
In particular, the cohomological grading makes $\Hbul(A,k)$ into a graded-commutative superalgebra.
\end{corollary}

\begin{remark}
The last statement of \cref{cor:gradedcommutativity} is also a consequence of \cite[Corollary 3.13]{Mastnak:2010}.
\end{remark}

Since $\Hbul(A,k)$ is a graded-commutative superalgebra, we can make the following definitions:

\begin{definition}
Let $A$ be a Hopf superalgebra, and let $M$ and $N$ be left $A$-supermodules. Let $I_A(M,N)$ be the annihilator ideal for the left cup product action of $\Hbul(A,k)$ on $\Ext_A^\bullet(M,N)$, and set $I_A(M) = I_A(M,M) = \ker(\Phi_M)$. Define the cohomology variety $\abs{A}$ of $A$ by
\[
\abs{A} = \Max \left( \Hbul(A,k) \right).
\]
Given left $A$-supermodules $M$ and $N$, define the relative support variety $\abs{A}_{(M,N)}$ by
\[
\abs{A}_{(M,N)} = \Max \left( \Hbul(A,k)/I_A(M,N) \right),
\]
and define the support variety $\abs{A}_M$ by
\[
\abs{A}_{M} = \abs{A}_{(M,M)} = \Max \left( \Hbul(A,k)/I_A(M) \right).
\]
\end{definition}

\begin{proposition} \label{standardproperties}
Let $A$ be a Hopf superalgebra and let $M$ and $N$ be $A$-supermodules. Then:
	\begin{enumerate}
	\item $\abs{A}_{M \otimes N} \subseteq \abs{A}_M$. If $A$ is cocommutative, then $\abs{A}_{M \otimes N} \subseteq \abs{A}_M \cap \abs{A}_N$.\label{item:tensorintersection}
	
	\item $\abs{A}_{(M,N)} \subseteq \abs{A}_M \cap \abs{A}_N$.\label{item:relativesupport}

	\item $\abs{A}_{(\Pi(M),N)} = \abs{A}_{(M,N)} = \abs{A}_{(M,\Pi(N))}$. In particular, $\abs{A}_{\Pi(M)} = \abs{A}_M$. \label{item:parity}
	
	\item $\abs{A}_{M \oplus N} = \abs{A}_M \cup \abs{A}_N$. \label{item:directsum}
	\end{enumerate}
\end{proposition}

\begin{proof}
First, from the interpretation of homogeneous elements in $\Hbul(A,k)$ as equivalence classes of length-$n$ homogeneous exact sequences, it is clear that $\Phi_{M \otimes N}: \Hbul(A,k) \rightarrow \Ext_A^\bullet(M \otimes N,M \otimes N)$ factors through the map $\Phi_M: \Hbul(A,k) \rightarrow \Ext_A^\bullet(M,M)$, and hence $\abs{A}_{M \otimes N} \subseteq \abs{A}_M$. If in addition $A$ is cocommutative, then the supertwist map defines an $A$-supermodule isomorphism $M \otimes N \cong N \otimes M$. Then $\abs{A}_{M \otimes N} = \abs{A}_{N \otimes M} \subseteq \abs{A}_N$, so $\abs{A}_{M \otimes N} \subseteq \abs{A}_M \cap \abs{A}_N$.

Next, let $\alpha \in \Hbul(A,k)$ and $\beta \in \Ext_A^\bullet(M,N)$ be homogeneous. By \cref{cupcomposition}, $\alpha \cup \beta = \Phi_N(\alpha) \circ \beta = \pm (\beta \circ \Phi_M(\alpha))$. Then $I_A(M,N) \supseteq I_A(M) + I_A(N)$, so $\abs{A}_{(M,N)} \subseteq \abs{A}_M \cap \abs{A}_N$. For the equalities in (\ref{item:parity}), observe that the odd isomorphisms \eqref{eq:homparityflip} induce odd isomorphisms
\begin{equation} \label{eq:extparityflip}
\Ext_A^\bullet(M,N) \simeq \Ext_A^\bullet(\Pi(M),N) \quad \text{and} \quad \Ext_A^\bullet(M,N) \simeq \Ext_A^\bullet(M,\Pi(N)) 
\end{equation}
that are compatible with left Yoneda multiplication by $\Ext_A^\bullet(N,N)$ and right Yoneda multiplication by $\Ext_A^\bullet(M,M)$, respectively (cf.\ \cite[Remark 3.2.2]{Drupieski:2016}). Then the equality $\alpha \cup \beta = \Phi_N(\alpha) \circ \beta$ implies that $I_A(M,N) = I_A(\Pi(M),N)$ while the equality $\alpha \cup \beta = \pm (\beta \circ \Phi_M(\alpha))$ implies that $I_A(M,N) = I_A(M,\Pi(N))$. Then $\abs{A}_{(\Pi(M),N)} = \abs{A}_{(M,N)} = \abs{A}_{(M,\Pi(N))}$.

Finally, observe that $\Ext_A^\bullet(M \oplus N,M \oplus N)$ decomposes into the direct sum of $\Ext_A^\bullet(M,M)$, $\Ext_A^\bullet(N,N)$, $\Ext_A^\bullet(M,N)$, and $\Ext_A^\bullet(N,M)$. This implies that
\[
I_A(M \oplus N) = I_A(M) \cap I_A(N) \cap I_A(M,N) \cap I_A(N,M).
\]
But $I_A(M,N)$ and $I_A(N,M)$ both contain $I_A(M) + I_A(N)$ as observed in the previous paragraph, so $I_A(M,N) = I_A(M) \cap I_A(N)$. This implies that $\abs{A}_{M \oplus N} = \abs{A}_M \cup \abs{A}_N$.
\end{proof}

\begin{remark}
If $A$ is not cocommutative, one need not have $\abs{A}_{M \otimes N} \subseteq \abs{A}_M \cap \abs{A}_N$; see \cite{Benson:2014,Feldvoss:2015}.
\end{remark}

\begin{proposition} \label{relativeses}
Let $A$ be a Hopf superalgebra, and let $M$ be an $A$-supermodule.
	\begin{enumerate}
	\item Let $0 \rightarrow M_1 \rightarrow M_2 \rightarrow M_3 \rightarrow 0$ be a short exact sequence in $(\fsmod_A)_\ev$. Then
	\[
	\abs{A}_{(M,M_r)} \subseteq \abs{A}_{(M,M_s)} \cup \abs{A}_{(M,M_t)} \quad \text{and} \quad \abs{A}_{(M_r,M)} \subseteq \abs{A}_{(M_s,M)} \cup \abs{A}_{(M_t,M)}
	\]
	whenever $\set{r,s,t} = \set{1,2,3}$. \label{item:2outof3}
	
	\item \label{item:relativesimple} Suppose $M$ is finite-dimensional. Then
	\[
	\abs{A}_M = \bigcup_S \abs{A}_{(M,S)} = \bigcup_S \abs{A}_{(S,M)},
	\]
	where the unions are taken over the $A$-supermodule composition factors of $A$.
	\end{enumerate}
\end{proposition}

\begin{proof}
Let $0 \rightarrow M_1 \rightarrow M_2 \rightarrow M_3 \rightarrow 0$ be a short exact sequence in $(\fsmod_A)_\ev$. Then there exists a corresponding long exact sequence 
\[
\cdots \rightarrow \Ext_A^n(M,M_1) \rightarrow \Ext_A^n(M,M_2) \rightarrow \Ext_A^n(M,M_3) \rightarrow \Ext_A^{n+1}(M,M_1) \rightarrow \cdots
\]
whose connecting homomorphisms are given by left Yoneda multiplication by the extension class in $\Ext_A^1(M_3,M_1)$ of the original short exact sequence. In particular, the connecting homomorphisms commute with the right cup product action of $\Hbul(A,k)$, as do the other maps in the sequence. Now from the exactness of the sequence it follows that if $\set{r,s,t} = \set{1,2,3}$, then $I_A(M,M_r) \supseteq I_A(M,M_s) \cdot I_A(M,M_t)$, and hence $\abs{A}_{(M,M_r)} \subseteq \abs{A}_{(M,M_s)} \cup \abs{A}_{(M,M_t)}$. The second stated inclusion in (\ref{item:2outof3}) follows via a similar argument by considering the long exact sequence for $\Ext$ in the first variable. Now suppose $M$ is finite-dimensional. Arguing by induction on the number of $A$-supermodule composition factors in $M$, the inclusions $\abs{A}_M \subseteq \bigcup_S \abs{A}_{(M,S)}$ and $\abs{A}_M \subseteq \bigcup_S \abs{A}_{(S,M)}$ follow from (\ref{item:2outof3}), while the reverse inclusions follow from \cref{standardproperties}(\ref{item:relativesupport}).
\end{proof}

\begin{definition} \label{def:complexity}
Let $M$ be an $A$-supermodule, and let
\[
\cdots \rightarrow P_2 \rightarrow P_1 \rightarrow P_0 \rightarrow M \rightarrow 0
\]
be a minimal projective resolution of $M$ in $(\fsmod_A)_\ev$. The complexity of $M$ as an $A$-supermodule, denoted $\cx_A(M)$, is the least element $s \in \N \cup \set{\infty}$ for which there exists a positive real number $\kappa$ such that $\dim P_n \leq \kappa \cdot n^{s-1}$ for all $n > 0$. We declare $\cx_{A}(M)=\infty$ if no such $s$ exists. Equivalently, $\cx_A(M)$ is the complexity of $M$ as a module for the smash product algebra $A \# k\Z_2$, i.e., $\cx_A(M) = \cx_{A \# k\Z_2}(M)$.
\end{definition}

In the next proposition we introduce the assumption that $k$ is algebraically closed so that we can apply results from the literature on support varieties and the dimensions of affine varieties.

\begin{proposition} \label{dimensioncomplexity}
Let $k$ be an algebraically closed field of characteristic $p \neq 2$, and let $A$ be a finite-dimensional $k$-Hopf superalgebra. Suppose that $\Hbul(A,k)$ is a finitely-generated superalgebra, and suppose for each pair of finite-dimensional $A$-supermodules $M$ and $N$ that the cup product makes $\Ext_A^\bullet(M,N)$ into a finitely-generated left $\Hbul(A,k)$-module. Then
\[
\dim \abs{A}_M = \cx_A(M).
\]
In particular, $\abs{A}_M = \set{0}$ if and only if $M$ is projective.
\end{proposition}

\begin{proof}
The hypotheses imply that the subring $\Hbul(A,k)_{\zero}$ of $\Hbul(A,k)$ is a finitely-generated algebra, and if $M$ and $N$ are finite-dimensional $A$-supermodules, then $\Ext_A^\bullet(M,N)_{\zero}$ is a finitely-generated $\Hbul(A,k)_{\zero}$-module. More generally, the hypotheses imply that $\Hbul(A,k)$ and $\Ext_A^\bullet(M,N)$ are finitely-generated as modules over $\Hbul(A,k)_{\zero}$. Reinterpreting the hypotheses in terms of the smash product algebra $A \# k\Z_2$, which is a Hopf algebra in the ordinary sense, $\Hbul(A \# k\Z_2,k)$ is a finitely-generated algebra and $\Ext_{A \# k\Z_2}^\bullet(M,N)$ is a finite $\Hbul(A \# k\Z_2,k)$-module. Then by \cite[Proposition 3.2.7]{Kulshammer:2012} (see also \cite{Feldvoss:2010,Feldvoss:2015}, though Feldvoss and Witherspoon's definition of the action of a Hopf algebra on a space of linear maps is different from ours),
\[
\cx_A(M) = \cx_{A \# k\Z_2}(M) = \dim \abs{A \# k\Z_2}_M.
\]
Next, restriction from $A \# k\Z_2$ to $A$ induces an algebra homomorphism
\[
\varphi: \Hbul(A \# k\Z_2,k)/I_{A \#k\Z_2}(M) \rightarrow \Hbul(A,k)/I_A(M).
\]
This homomorphism identifies with the inclusion into $\Hbul(A,k)/I_A(M)$ of its even subspace. Since $\Hbul(A,k)$ is finite over $\Hbul(A,k)_{\zero}$ by the hypotheses of the proposition, $\Hbul(A,k)/I_A(M)$ is finite over the image of $\varphi$. Then $\dim \abs{A \# k\Z_2}_M = \dim \abs{A}_M$ because Krull dimension is preserved under integral ring extensions. So $\cx_A(M) = \dim \abs{A}_M$.
\end{proof}

\section{Finite supergroup schemes in characteristic zero} \label{section:char0}

In this section let $k$ be an algebraically closed field of characteristic $0$.

\subsection{Structure of finite supergroup schemes in characteristic zero}

The following theorem \cite[Theorem 3.3]{Kostant:1977}\footnote{As printed, the cited theorem contains a typo: `commutative' should be `cocommutative.'} describes the structure of cocommutative Hopf superalgebras over $k$.

\begin{theorem}[Kostant] \label{Kostant}
Let $k$ be an algebraically closed field of characteristic zero and let $A$ be a cocommutative Hopf super\-algebra over $k$. Let $G$ be the group of group-like elements in $A$, let $\g$ be the Lie superalgebra of primitive elements in $A$, and let $U(\g)$ be the universal enveloping superalgebra of $\g$. Then $A$ is isomorphic as a Hopf superalgebra to the smash product $U(\g) \# kG$ formed with respect to the homomorphism $\pi: G \rightarrow GL(\g)$ defined by $\pi(g)(x) = gxg^{-1}$.
\end{theorem}

Since $U(\g)$ is infinite-dimensional whenever $\gzero \neq 0$, we immediately get:

\begin{corollary} \label{cor:Kostant}
Let $k$ be an algebraically closed field of characteristic zero and let $A$ be a finite-dimensional cocommutative Hopf superalgebra over $k$. Then there exists a finite group $G$, a finite-dimensional odd superspace $V$, and a representation of $G$ on $V$ such that $A$ is isomorphic as a Hopf superalgebra to the smash product algebra $\Lambda(V) \# kG$.
\end{corollary}

Thus, the data of a finite supergroup scheme over an algebraically closed field $k$ of characteristic zero is equivalent to the data of a finite group $G$ and a purely odd finite-dimensional $kG$-module $V$. Given this data, we will denote the corresponding finite supergroup scheme by $V \rtimes G$, and we will identify the category of $V \rtimes G$-supermodules with the category of $\Lambda(V) \# kG$-supermodules. In particular, the support variety $\abs{V \rtimes G}_M$ of a $V \rtimes G$-supermodule $M$ is the support variety $\abs{\Lambda(V) \# kG}_M$ of $M$ as a $\Lambda(V) \# kG$-supermodule.

\subsection{Support varieties }

From \cref{cor:Kostant} we can describe the cohomology ring of any finite supergroup scheme over an algebraically closed field of characteristic zero.

\begin{theorem} \label{cohomologychar0}
Let $k$ be an algebraically closed field of characteristic zero. Let $G$ be a finite group, let $V$ be a purely odd finite-dimensional $kG$-module, and let $M$ and $N$ be finite-dimensional $\Lambda(V) \# kG$-supermodules. Then restriction from $\Lambda(V) \# kG$ to $\Lambda(V)$ induces a natural isomorphism
\begin{equation} \label{eq:resMN}
\Ext_{V \rtimes G}^\bullet(M,N) = \Ext_{\Lambda(V) \# kG}^\bullet(M,N) \cong \Ext_{\Lambda(V)}^\bullet(M,N)^G.
\end{equation}
In particular, restriction to $\Lambda(V)$ induces an isomorphism of graded superalgebras
\begin{equation} \label{eq:resk}
\Hbul(V \rtimes G,k) = \Hbul(\Lambda(V) \# kG,k) \cong \Hbul(\Lambda(V),k)^G \cong S^\bullet(V^*)^G,
\end{equation}
\end{theorem}

\begin{proof}
Set $A = \Lambda(V) \# kG$, and write $\Lambda^+(V)$ for the augmentation ideal of $\Lambda(V)$. Then $\Lambda(V)$ is a Hopf sub-superalgebra of $A$, and $\Lambda^+(V)$ generates a two-sided Hopf ideal in $A$. Let $A//\Lambda(V)$ denote the quotient of $A$ by the two-sided ideal in $A$ generated by $\Lambda^+(V)$. Then $A//\Lambda(V) \cong kG$ as Hopf superalgebras, and there exists a Lyndon--Hochschild--Serre spectral sequence
\begin{equation}
E_2^{i,j}(M,N) = \opH^i(kG, \opH^j(\Lambda(V),N \otimes M^*)) \Rightarrow \opH^{i+j}(A,N \otimes M^*).
\end{equation}
In particular, $E_r(k,k)$ is a spectral sequence of (super)algebras and $E_r(M,N)$ is naturally a left $E_r(k,k)$-(super)module for each $r \geq 2$.

The group algebra $kG$ is semisimple by the assumption that $k$ is a field of characteristic $0$, so $E_2^{i,j} = 0$ for all $i > 0$. Then the spectral sequence collapses to the row $i=0$. This implies that restriction to $\Lambda(V)$ induces isomorphisms
\[
\Ext_A^\bullet(M,N) \cong \Hbul(A,N \otimes M^*) \cong \Hbul(\Lambda(V),N \otimes M^*)^G \cong \Ext_{\Lambda(V)}^\bullet(M,N)^G,
\]
where the first and last isomorphisms in this string are consequences of the tensor-hom adjunction in \cref{tensorhomadjunction}. In particular, restriction to $\Lambda(V)$ induces an algebra isomorphism
\[
\Hbul(\Lambda(V) \# kG,k) \cong \Hbul(\Lambda(V),k)^G.
\]
Finally, it is a classical result \cite[2.2(2)]{Priddy:1970} that $\Hbul(\Lambda(V),k)$ is isomorphic to the symmetric algebra $S^\bullet(V^*)$ with $V^*$ concentrated in cohomological degree $1$.
\end{proof}

\begin{theorem}
Let $G$ be a finite group and let $V$ be a finite-dimensional purely odd $kG$-module. Then there exist isomorphisms of varieties
\[
\abs{V \rtimes G} \cong \Max(S(V^*)^G) \cong V/G
\]
\end{theorem}

\begin{proof}
This is an immediate application of \cite[Proposition 5.4.8]{Benson:1998a}.
\end{proof}

Given a $kG$-module $V$, write $[v]$ for the $G$-orbit of an element $v \in V$.

\begin{theorem} \label{T:char0rankvariety}
Let $G$ be a finite group and let $V$ be a finite-dimensional purely odd $kG$-module. Let $M$ be a finite-dimensional $V \rtimes G$-supermodule. Then
\[
\abs{V \rtimes G}_M \cong \set{[v] \in V/G : M|_{\subgrp{v}} \text{ is not free}}
\]
\end{theorem}

\begin{proof}
Set $E = \Lambda(V)$. Let $\Phi: \Hbul(E,k) \rightarrow \Ext_E^\bullet(M,M)$ be the natural ring homomorphism, and set $I = \ker(\Phi)$. Then $\abs{E}_M = \Max(\Hbul(E,k)/I)$. Identifying $\Hbul(E,k)$ with the symmetric algebra $S(V^*)$, it follows from the results in \cite[\S3]{Aramova:2000} that $\abs{E}_M \cong \set{v \in V : M|_{\subgrp{v}} \text{ is not free}}$. More precisely, let $I'$ be the annihilator ideal for the left Yoneda product action of $\Hbul(E,k) = \Ext_E^\bullet(k,k)$ on $\Ext_E^\bullet(M,k)$. Then \cite[Theorem 3.9]{Aramova:2000} asserts that
\[
\abs{E}_{(M,k)} = \Max(\Hbul(E,k)/I') \cong \set{v \in V : M|_{\subgrp{v}} \text{ is not free}}.
\]
Up to parity shift and isomorphism, the trivial module $k$ is the unique irreducible $E$-supermodule, so \cref{standardproperties,relativeses} imply that $\abs{E}_{(M,k)} = \abs{E}_M$.

Next, the map $\Phi$ is a $G$-module homomorphism. Since $G$ is a finite group and since $k$ is a field of characteristic zero, the fixed point functor $(-)^G$ is exact. Then one obtains an exact sequence
\[
0 \longrightarrow I^G \longrightarrow \Hbul(E,k)^G \stackrel{\Phi^G}{\longrightarrow} \im(\Phi_M)^G \longrightarrow 0,
\]
where $\Phi^G : \Hbul(E,k)^G \rightarrow \Ext_E^\bullet(M,M)^G$ is the map induced by $\Phi$. In particular, $I^G = \ker(\Phi^G)$, and $\Hbul(E,k)^G/I^G \cong (\Hbul(E,k)/I)^G$ as algebras. Now \cref{cohomologychar0} implies that
\[
\abs{V \rtimes G}_M \cong \Max\left( \Hbul(E,k)^G/\ker(\Phi^G) \right) \cong \Max\left( (\Hbul(E,k)/I)^G \right).
\]
Then by \cite[Proposition 5.4.8]{Benson:1998a}, the support variety $\abs{V \rtimes G}_M$ identifies with the quotient of $\abs{E}_M$ by the action of $G$, i.e., $\abs{V \rtimes G}_M \cong \set{[v] \in V/G : M|_{\subgrp{v}} \text{ is not free}}$.
\end{proof}

As a corollary of the ``rank variety'' description in \cref{T:char0rankvariety} we get the following tensor product property:

\begin{corollary} \label{tensorproducttheorem}
Let $G$ be a finite group and let $V$ be a finite-dimensional purely odd $kG$-module. Let $M$ and $N$ be finite-dimensional $V \rtimes G$-supermodules. Then
\[
\abs{V \rtimes G}_{M\otimes N} = \abs{V \rtimes G}_M \cap \abs{V \rtimes G}_N.
\]
\end{corollary}

\begin{proof}
We use the description of support varieties provided by \cref{T:char0rankvariety}. Since $\Lambda(V) \# kG$ is a cocommutative Hopf superalgebra, one has
\[
\abs{V \rtimes G}_{M\otimes N} \subseteq \abs{V \rtimes G}_M \cap \abs{V \rtimes G}_N
\]
by \cref{standardproperties}(\ref{item:tensorintersection}). Now let $[v] \in \abs{V \rtimes G}_M \cap \abs{V \rtimes G}_N$. Viewing $V$ as an abelian purely odd Lie superalgebra, the subalgebra $\subgrp{v}$ of $\Lambda(V) \# kG$ generated by $v$ is isomorphic to a one-variable exterior algebra, $\Lambda(v)$. Up to isomorphism and parity change, the only indecomposable $\Lambda(v)$-supermodules are the trivial module and its projective cover (see, e.g., \cite[Proposition 5.2.1]{Boe:2010}). Consequently, a $\subgrp{v}$-supermodule is not free if and only if it contains the trivial module as a direct summand. In particular, when written as a direct sum of indecomposable $\subgrp{v}$-supermodules, both $M$ and $N$ must have a trivial direct summand and hence so must $M \otimes N$. Thus, $M\otimes N$ is not free as a $\subgrp{v}$-supermodule and  $[v] \in \abs{V \rtimes G}_{M\otimes N}$. 
\end{proof}

\begin{theorem} \label{T:twodivisibility}
Let $G$ be a finite group and let $V$ be a finite-dimensional purely odd $kG$-module. Let $M$ be a finite-dimensional $V \rtimes G$-supermodule and let $d = \dim \abs{V \rtimes G} - \dim \abs{V \rtimes G}_M$ be the codimension of $\abs{V \rtimes G}_M$ in $\abs{V \rtimes G}$. Then $2^d \mid \dim_k M$. If $d >0$, then $\sdim(M) = 0$.
\end{theorem}

\begin{proof}
Set $E = \Lambda(V)$. By the proof of \cref{T:char0rankvariety}, $\abs{V \rtimes G}_M$ identifies with the quotient of $\abs{E}_M$ by the action of $G$. Since $G$ is a finite group, then the codimension of $\abs{V \rtimes G}_M$ in $\abs{V \rtimes G}$ is equal to the codimension of $\abs{E}_M$ in $\abs{E}$. Next, $\abs{E}_M$ is defined by an ideal in the polynomial ring $\Hbul(E,k) \cong S(V^*)$, so it follows from the Noether Normalization Theorem \cite[Theorem II.3.1]{Kunz:2013} that there exists a $d$-dimensional subspace $H$ of $V$ such that $H \cap \abs{E}_M = \set{0}$. Considering the restriction of $M$ to the subalgebra $\Lambda(H)$ of $E$ generated by $H$, this implies that $\abs{\Lambda(H)}_M = \set{0}$, and hence implies by \cref{dimensioncomplexity} that $M$ is projective as a $\Lambda(H)$-supermodule. Since $\Lambda(H)$ is indecomposable over itself, $M$ is then isomorphic as a $\Lambda(H)$-supermodule to a direct sum of copies of $\Lambda(H)$, each of which is of dimension $2^d$ and (if $d > 0$) of superdimension $0$.
\end{proof}

Let $V \rtimes G$-mod denote the category of all finite-dimensional $V \rtimes G$-modules and let $\mathcal{K}$ denote the stable module category of $V \rtimes G$-mod.  Using \cref{L:selfinjective}, it follows that $\mathcal{K}$ is a tensor triangulated category. We expect that the spectrum of $\mathcal{K}$ (in the sense of Balmer \cite{Balmer:2005}) is homeomorphic to $V/G$ and that the thick tensor ideals of $\mathcal{K}$ are classified by the specialization closed subsets of $V/G$ via support varieties; see the results of Pevtsova and Witherspoon \cite{Pevtsova:2015} in a setting with close similarities to the one considered here.

\section{Finite-dimensional Lie superalgebras} \label{sec:fdLSAs}

Now let $k$ be an algebraically closed field of characteristic $p \geq 3$, and let $\g$ be a finite-dimensional Lie superalgebra over $k$. In this section we investigate support varieties for the (typically infinite-dimensional) Hopf superalgebra $U(\g)$, the universal enveloping superalgebra of $\g$.

\subsection{Background on Lie superalgebra cohomology} \label{subsec:Liebackground}

Write $\Lambda_s(\g^*)$ for the superexterior algebra on $\g^*$. We consider $\Lambda_s(\g^*)$ as a graded superalgebra with $\g^*$ as concentrated in $\Z$-degree $1$. Then $\Lambda_s(\g^*)$ identifies with the graded tensor product of algebras $\Lambda(\gzero^*) \gotimes S(\gone^*)$. Equivalently, $\Lambda_s(\g^*)$ is the free graded-commutative graded superalgebra generated by $\g^*$. Now let $M$ be a finite-dimensional $\g$-supermodule. There exists a differential $\partial$ on $C(\g,M) := M \otimes \Lambda_s(\g^*)$, called the Koszul differential, that makes $C(\g,M)$ into a cochain complex. The Koszul differential acts by derivations on $C(\g,M)$, i.e.,
\begin{align}
\partial(m \otimes z) &= \partial(m) z + m \otimes \partial(z) & \text{if $m \in M$ and $z \in \Lambda_s(\g^*)$, and} \\
\partial(ab) &= \partial(a) \cdot b + (-1)^i a \cdot \partial(b) & \text{if $a \in \Lambda_s^i(\g^*)$ and $b \in \Lambda_s(\g^*)$.}
\end{align}
Thus, $\partial$ is determined by its actions on $M$ and $\Lambda_s^1(\g^*) = \g^*$. The map $\partial: M \rightarrow M \otimes \Lambda^1(\g^*) = M \otimes \g^*$ satisfies $\partial(m) = \sum_i m_i \otimes f_i$, where the $m_i \in M$ and $f_i \in \g^*$ are such that $\sum_i f_i(z).m_i = (-1)^{\ol{z} \cdot \ol{m}} z.m$ for each $z \in \g$. Next, since $p \geq 3$ there exists a natural isomorphism $\Lambda_s^2(\g^*) \cong [\Lambda_s^2(\g)]^*$.\footnote{The graded dual of $\Lambda_s(\g)$ is isomorphic as a graded superalgebra to $\bsa(\g^*)$, the superalgebra of (super)alternating tensors on $\g^*$; see \cite[\S2.6]{Drupieski:2016}. The identification $\Lambda_s^1(\g^*) = \g^* = \bsa^1(\g^*)$ extends to a homomorphism of graded superalgebras $\psi: \Lambda_s(\g^*) \rightarrow \bsa(\g^*)$; in characteristic zero this is an isomorphism, but in characteristic $p \geq 3$ it need only be an isomorphism in $\Z$-degrees less than $p$. In particular, $\Lambda_s^2(\g^*) \cong \bsa^2(\g^*) \cong [\Lambda_s^2(\g)]^*$.} Then $\partial: \Lambda_s^1(\g^*) \rightarrow \Lambda_s^2(\g^*)$ identifies with the linear map $\g^* \rightarrow [\Lambda_s^2(\g)]^*$ that is the transpose of the Lie bracket $\Lambda_s^2(\g) \rightarrow \g$, $z_1 \wedge z_2 \mapsto [z_1,z_2]$. The cohomology of $C(\g,M)$ with respect to $\partial$ is the Lie superalgebra cohomology group $\Hbul(\g,M)$. For more details, see \cite[\S\S3.1--3.2]{Drupieski:2013}.

\begin{example} \label{abelianodd}
If $\g = \gone$, then $\g$ is abelian and $U(\g) = \Lambda(\g)$, so $\Hbul(\g,k) \cong S(\g^*)$ by \cite[2.2(2)]{Priddy:1970}.
\end{example}

\begin{example} \label{smallestcalculation}
Let $\g$ be the Lie superalgebra over $k$ generated by an odd element $y$ and an even element $x$ such that $y^2 := \frac{1}{2}[y,y] = x$. This implies that $[x,y] = 0$, and hence that $\set{x,y}$ is a homogeneous basis for $\g$. Let $\set{x^*,y^*}$ be the corresponding dual basis. Then in the notation of \cite[\S2.3]{Drupieski:2013}, a typical homogeneous monomial in $\Lambda(\g^*)$ has the form $\subgrp{(x^*)^a}(y^*)^b$ with $a \in \set{0,1}$ and $b \in \N$. Now the Koszul differential $\partial: \Lambda_s(\g^*) \rightarrow \Lambda_s(\g^*)$ satisfies $\partial(y^*) = 0$ and $\partial(\subgrp{x^*}) = -(y^*)^2$, so the cohomology ring $\Hbul(\g,k)$ is isomorphic to the truncated polynomial ring $k[y^*]/\subgrp{(y^*)^2}$.
\end{example}

\subsection{The cohomological spectrum}

Write $\gone^*[p]$ for the superspace $\gone^*$ considered as a graded superspace concentrated in $\Z$-degree $p$. Then $S(\gone^*[p])$ is a graded superalgebra concentrated in $\Z$-degrees divisible by $p$. Since $k$ is perfect, and since $S(\gone^*)$ is commutative in the ordinary sense, the $p$-power map, $z \mapsto z^p$, induces an injective homomorphism of graded superalgebras $S(\gone^*[p])^{(1)} \rightarrow S(\gone^*)$. Composing with the natural inclusion $S(\gone^*) \hookrightarrow \Lambda_s(\g^*)$, this produces an injective homomorphism of graded superalgebras
\begin{equation} \label{eq:ppowermap}
\wt{\varphi}: S(\gone^*[p])^{(1)} \rightarrow \Lambda_s(\g^*).
\end{equation}
Since $\Lambda_s(\g^*)$ is graded-commutative, and since the Koszul differential acts by derivations on $\Lambda_s(\g^*)$, it follows that the image of $\wt{\varphi}$ consists of cocycles in $C(\g,k)$, and hence that $\wt{\varphi}$ induces a graded superalgebra homomorphism
\begin{equation} \label{eq:varphi}
\varphi: S(\gone^*[p])^{(1)} \rightarrow \Hbul(\g,k).
\end{equation}

\begin{definition}
Given finite-dimensional $\g$-supermodules $M$ and $N$, let $I_\g(M,N)$ be the annihilator ideal for the cup product action of $\Hbul(\g,k)$ on $\Ext_\g^\bullet(M,N)$, and set $J_\g(M,N) = \varphi^{-1}(I_\g(M,N))$. Define $\Chi_\g(M,N)$ to be the subvariety of $\Max( S(\gone^*[p])^{(1)} )$ defined by $J_\g(M,N)$,
\[
\Chi_\g(M,N) = \Max \left( S(\gone^*[p])^{(1)}/J_\g(M,N) \right).
\]
Set $J_\g(M) = J_\g(M,M)$, and set $\Chi_\g(M) = \Chi_\g(M,M) = \Max ( S(\gone^*[p])^{(1)}/J_\g(M) )$.
\end{definition}

\begin{proposition} \label{ontopthpowers}
The induced map $\ol{\varphi}: S(\gone^*[p])^{(1)} \rightarrow \Hbul(\g,k)/\Nil(\Hbul(\g,k))$ is surjective onto all $p$-th powers.
\end{proposition}

\begin{proof}
Let $r \in \Hbul(\g,k)$. We want to show that the coset $r^p + \Nil(\Hbul(\g,k))$ is in the image of $\ol{\varphi}$. Since we are working modulo the nilradical of $\Hbul(\g,k)$, we may assume by \cref{cor:nilradical} that $r = r_0^0 + r_1^1$, where $r_0^0 \in \opH^\ev(\g,k)_{\zero}$ and $r_1^1 \in \opH^\odd(\g,k)_{\one}$. Since $r_0^0$ and $r_1^1$ commute in the ordinary sense by the fact that $\Hbul(\g,k)$ is a graded-commutative superalgebra, one has $r^p = (r_0^0)^p + (r_1^1)^p$.

The ring $\Hbul(\g,k)$ inherits its algebra structure from $\Lambda_s(\g^*) \cong \Lambda(\gzero^*) \gotimes S(\gone^*)$. Then there exist homogeneous elements $a_i,c_i \in \Lambda(\gzero^*)$ and $b_i,d_i \in S(\gone^*)$ such that $\sum a_i \otimes b_i$ is a cocycle representative for $r_0^0$ and $\sum c_i \otimes d_i$ is a cocycle representative for $r_1^1$. Since $r_0^0 \in \opH^\ev(\g,k)_{\zero}$, we may assume for each $i$ that $b_i \in S^\ev(\gone^*) = S(\gone^*)_{\zero}$ and $a_i \in \Lambda^\ev(\gzero^*)$. Similarly, since $r_1^1 \in \opH^\odd(\g,k)_{\one}$, we may assume for each $i$ that $d_i \in S^\odd(\gone^*) = S(\gone^*)_{\one}$ and $c_i \in \Lambda^\ev(\gzero^*)$. These assumptions imply, by the graded-commutativity of $\Lambda_s(\g^*)$, that the summands in $\sum a_i \otimes b_i$ pairwise commute in $\Lambda_s(\g^*)$ in the ordinary sense and that $a_i$ and $b_i$ commute in $\Lambda_s(\g^*)$ in the ordinary sense. Similarly, the summands in $\sum c_i \otimes d_i$ pairwise commute, and $c_i$ commutes with $d_i$. Then
\begin{align*}
\left(\sum a_i \otimes b_i \right)^p &= \sum (a_i \otimes b_i)^p = \sum a_i^p \otimes b_i^p, \quad \text{and} \\
\left(\sum c_i \otimes d_i \right)^p &= \sum (c_i \otimes d_i)^p = \sum c_i^p \otimes d_i^p.
\end{align*}
Since $\Lambda(\gzero^*)$ is graded-commutative and is generated as an algebra by the subspace $\Lambda^1(\gzero^*)$, which is concentrated in $\Z$-degree $1$, it follows that every homogeneous element in $\Lambda(\gzero^*)$ of nonzero $\Z$-degree squares to $0$. This implies that the only nonzero summands appearing in $\sum a_i^p \otimes b_i^p$ and $\sum c_i^p \otimes d_i^p$ are those for which $a_i,c_i \in \Lambda^0(\gzero^*)$, i.e., those for which $a_i$ and $c_i$ are scalars. But if $a_i$ and $c_i$ are scalars, then $a_i^p \otimes b_i^p$ and $c_i^p \otimes d_i^p$ are in the image of $\wt{\varphi}$. Now since $\sum a_i^p \otimes b_i^p$ and $\sum c_i^p \otimes d_i^p$ are cocycle representatives for $(r_0^0)^p$ and $(r_1^1)^p$, respectively, this shows that the coset of $r^p = (r_0^0)^p + (r_1^1)^p$ is in the image of $\ol{\varphi}$.
\end{proof}

\begin{corollary} \label{cor:homeomorphism}
The map $\varphi: S(\gone^*[p])^{(1)} \rightarrow \Hbul(\g,k)$ induces for each pair of finite-dimensional $\g$-supermodules $M$ and $N$ a homeomorphism $\varphi_{M,N}^*: \abs{U(\g)}_{(M,N)} \simrightarrow \Chi_\g(M,N)$.
\end{corollary}

\begin{proof}
The algebra homomorphism \eqref{eq:varphi} induces an injective algebra homomorphism
\[
\varphi_{M,N}: S(\gone^*[p])^{(1)}/J_\g(M,N) \rightarrow \Hbul(\g,k)/I_\g(M,N).
\]
\cref{ontopthpowers} implies that, modulo nilpotent elements, $\varphi_{M,N}$ is surjective onto $p$-th powers. Then the induced morphism between the maximal ideal spectra, $\varphi_{M,N}^*: \abs{U(\g)}_{(M,N)} \rightarrow \Chi_\g(M,N)$, is a bijection, hence a homeomorphism.
\end{proof}

Using \cref{cor:homeomorphism} we will restrict our attention to the varieties $\Chi_\g(M)$. Our first step will be to determine the variety $\Chi_\g(k)$. It identifies with the affine subvariety of $\gone$ defined by $\ker(\varphi)$.

\begin{theorem} \label{nonrestrictedspectrum}
Let $\g$ be a finite-dimensional Lie superalgebra over $k$. Then
\[
\Chi_\g(k) \cong \set{x \in \gone : [x,x] = 0}
\]
\end{theorem}

\begin{proof}
First let $a \in \g_{\ol{0}}^* \subseteq \Lambda_s^1(\g^*)$. Since the Koszul differential $\partial$ acts by derivations on $\Lambda_s(\g^*)$, one has $\partial(a \cdot \partial(a)^{p-1}) = \partial(a)^p$. Then $\partial(a)^p$ is a coboundary in $C(\g,k)$. Next, $\partial(a) \in \Lambda_s^2(\g^*)_{\zero} = \Lambda^2(\gzero^*) \oplus S^2(\gone^*)$. The elements of $\Lambda^2(\gzero^*)$ commute in the ordinary sense with the elements of $S^2(\gone^*)$. Then since the elements of $\Lambda^2(\gzero^*)$ square to zero, it follows that $\partial(a)^p \in S^{2p}(\gone^*)$. More precisely, it follows that $\partial(a)^p$ is an element in the image of the homomorphism $\wt{\varphi}: S(\gone^*[p])^{(1)} \rightarrow \Lambda_s(\g^*)$. Since $\wt{\varphi}$ is an injection, there exists a unique element $\phi_a \in S^2(\gone^*[p])^{(1)}$ such that $\wt{\varphi}(\phi_a) = \partial(a)^p$. Since $\partial(a)^p$ is a coboundary in $C(\g,k)$, this means that $\phi_a \in \ker(\wt{\varphi})$.

Now let $x \in \gone$. Considering $\phi_a$ and $\partial(a)^p \in S^{2p}(\gone^*)$ as polynomial functions on $\gone$, one has
\[
[\phi_a(x)]^p = [\partial(a)^p](x) = a([x,x])^p.
\]
Then $\phi_a(x) = 0$ if and only if $a([x,x]) = 0$. Varying $a$ over the elements of $\gzero^*$, it follows that $\Chi_\g(k) \subseteq \set{x \in \gone : [x,x] = 0}$. To prove the reverse equality, let $x \in \gone$ such that $[x,x] = 0$, and consider the abelian Lie sub-superalgebra $\fa$ of $\g$ spanned by $x$. Restriction from $\g$ to $\fa$ induces a natural morphism of varieties $\iota: \Chi_{\fa}(k) \rightarrow \Chi_\g(k)$. The restriction map $S(\gone^*[p])^{(1)} \rightarrow S(\fa_{\one}^*[p])^{(1)}$ is a surjection, so it follows that $\iota: \Chi_{\fa}(k) \rightarrow \Chi_\g(k)$ is a closed embedding. Since $\fa$ is a purely odd Lie superalgebra, the universal enveloping superalgebra of $\fa$ identifies with the exterior algebra $\Lambda(\fa)$. Then $\Hbul(\fa,k) = \Hbul(\Lambda(\fa),k) \cong S(\fa^*)$, so $\Chi_{\fa}(k) \cong \fa = \fa_1$. Now by the naturality of the morphism $\iota: \Chi_\fa(k) \rightarrow \Chi_\g(k)$, we conclude that $x \in \Chi_\g(k)$, and hence that $\Chi_\g(k) = \set{x \in \gone: [x,x] = 0}$.
\end{proof}

\subsection{Rank varieties}

We next show that support varieties in this setting admit a rank variety description.

\begin{theorem} \label{associatedvariety}
Let $\g$ be a finite-dimensional Lie superalgebra over $k$, and let $M$ be a finite-dimensional $\g$-supermodule. Then
\[
\Chi_\g(M) = \set{x \in \Chi_\g(k) : M|_{\subgrp{x}} \text{ is not free}} \cup \set{0}.
\]
\end{theorem}

\begin{proof}
Our argument is exactly parallel to that used by Jantzen \cite[\S3]{Jantzen:1986} to calculate the support varieties of finite-dimensional restricted Lie algebras in characteristic $2$.

Write $\Chi_\g'(M)$ for the set on the right-hand side of the equality stated in the theorem. Then we must show that $\Chi_\g(M) = \Chi_\g'(M)$. Given $x \in \Chi_\g(k)$, write $\subgrp{x}$ of the Lie sub-superalgebra of $\g$ generated by $x$. Since $x$ is a self-commuting odd element, the enveloping superalgebra $U(\subgrp{x})$ is a one-variable exterior algebra. To prove the inclusion $\Chi_\g'(M) \subseteq \Chi_\g(M)$, observe that if $x \in \Chi_\g(k)$ and if $M|_{\subgrp{x}}$ is not free, then $\Chi_{\subgrp{x}}(M|_{\subgrp{x}}) \cong \subgrp{x}$ by the results in \cite[\S3]{Aramova:2000}. This implies as in the proof of \cref{nonrestrictedspectrum} that $\subgrp{x} \subseteq \Chi_\g(M)$, and hence that $\Chi'_\g(M) \subseteq \Chi_\g(M)$. For the reverse inclusion $\Chi_\g(M) \subseteq \Chi'_\g(M)$, observe that $\Chi_\g'(M) = \Chi_\g(k)$ unless the superdimension of $M$ is zero, i.e., unless $M$ is isomorphic as a superspace to $k^{m|m}$ for some $m$. So let us assume that $M \cong k^{m|m}$.

First suppose that $\g = \gl(M)$. Fixing a homogeneous basis for $M$, $\g$ identifies with the general linear Lie superalgebra $\glmm$. Now let $x \in \Chi_\g(M)$ be nonzero, and suppose to the contrary that $M|_{\subgrp{x}}$ is free. Using this fact and choosing a new homogeneous basis for $M$ if necessary, we may write $x$ in the form shown in Figure~\ref{fig:xmatrix},
\begin{figure}[tbp]
\scalebox{0.8}{
$
\left[
\begin{array}{cccccc|cccccc}
	&	&	&	&	&	& 1	&	&	&	&	&	\\
	&	&	&	&	&	&	& \ddots	&	&	&	& \\
	&	&	&	&	&	&	&	& 1	&	&	& 	\\
	&	&	&	&	&	&	&	&	& 0	&	& 	\\
	&	&	&	&	&	&	&	&	&	& \ddots	& 	\\	
	&	&	&	&	&	&	&	&	&	&	& 0	\\
\hline
0	&	&	&	&	&	&	&	&	&	&	&	\\
	& \ddots	&	&	&	&	&	&	&	&	&	&	\\
	&	& 0	&	&	&	&	&	&	&	&	&	\\
	&	&	& 1	&	&	&	&	&	&	&	&	\\
	&	&	&	& \ddots	&	&	&	&	&	&	&	\\
	&	&	&	&	& 1	&	&	&	&	&	&
\end{array}
\right]
$
}
\caption{The matrix $x$ in the proof of \cref{associatedvariety}.} \label{fig:xmatrix}
\end{figure}
where, say, the first $n$ diagonal entries of the upper-right $m \times m$ block are equal to $1$, the last $m-n$ diagonal entries of the lower-left $m \times m$ block are equal to $1$, and the rest of the entries in $x$ are $0$. Let $\set{e_1,\ldots,e_m,e_{m+1},\ldots,e_{2m}}$ be the ``standard'' homogeneous basis for $k^{m|m}$ with respect to which the matrix for $x$ has been written (so $e_1,\ldots,e_m$ are even basis vectors and $e_{m+1},\ldots,e_{2m}$ are odd basis vectors), and let $\phi: k^{m|m} \rightarrow k^{m|m}$ be the linear map that interchanges $e_i$ and $e_{m+i}$ for $n+1 \leq i \leq m$ and that leaves the remaining basis vectors fixed. Then $\phi \circ \phi = 1$, and conjugation by $\phi$ defines a (non-homogeneous) automorphism of $\g$. This automorphism induces a (non-homogeneous) automorphism of $\Hbul(\g,k)$, which in turn induces an isomorphism of varieties $\Chi_\g(k^{m|m}) \cong \Chi_\g(k^{m|m})$ that sends $x$ to the odd matrix $y$, illustrated in Figure~\ref{fig:matrixy},
\begin{figure}[tbp]
\scalebox{0.8}{
$
\left[ \begin{array}{ccc|ccc}
	&	&	& 1	&	&	\\
	&	&	&	& \ddots	&	\\
	&	&	&	&	& 1	\\
\hline
\phantom{1}	&	&	&	&	&	\\
	& \phantom{\ddots}	&	&	&	&	\\
	&	& \phantom{1}	&	&	&
\end{array}
\right]
$
}
\caption{The matrix $y$ in the proof of \cref{associatedvariety}.} \label{fig:matrixy}
\end{figure}
whose upper-right block is the $m \times m$ identity matrix and whose remaining entries are $0$. To contradict the initial assumption that $M|_{\subgrp{x}}$ is free, we will show that $y \notin \Chi_\g(k^{m|m})$.

Let $\set{e_{ij}: 1 \leq i,j \leq 2m}$ be the homogeneous basis for $\glmm$ consisting of the usual matrix units (i.e., $e_{ij} \in \glmm$ has a $1$ in the $ij$-position and $0$s elsewhere), and let $\set{X_{ij}: 1 \leq i,j \leq 2m}$ be the corresponding dual basis. Write $\fS_m$ for the symmetric group on $m$ letters, and set
\[
f_1 = \sum_{\sigma \in \fS_m} \sgn(\sigma) \cdot X_{\sigma(1),m+1} X_{\sigma(2),m+2} \cdots X_{\sigma(m),2m} \in S(\gone^*).
\]
Considered as a polynomial function on $\gone^*$, $f_1(y) = 1 \neq 0$. We will show that $y \notin \Chi_\g(k^{m|m})$ by showing that $f_1 \in J_\g(k^{m|m})$.

By definition, $J_\g(k^{m|m})$ is the kernel of the composite ring homomorphism
\[
\varphi: S(\gone^*[p])^{(1)} \stackrel{\varphi}{\longrightarrow} \Hbul(\g,k) \stackrel{\Phi}{\longrightarrow} \Ext_\g^\bullet(k^{m|m},k^{m|m}).
\]
It follows from \cref{tensorhomadjunction} that $\Phi: \Hbul(\g,k) \rightarrow \Ext_\g^\bullet(k^{m|m},k^{m|m})$ may be rewritten as the ring homomorphism $\Hbul(\g,k) \rightarrow \Ext_\g^\bullet(k,\g)$ induced by the $\g$-supermodule homomorphism $k \rightarrow \g$ that sends a scalar $\lambda \in k$ to the corresponding scalar multiple of the identity matrix in $\g$. Next, the adjoint and coadjoint representations of $\g$ are isomorphic, i.e., $\g \cong \g^*$ as $\g$-supermodules. To see this, set $V = k^{m|m}$. Then $\g \cong V \otimes V^*$ and $V \cong (V^*)^*$ as $\g$-supermodules, so by \cref{tensorproductdual},
\[
\g \cong V \otimes V^* \cong (V^*)^* \otimes V^* \cong (V \otimes V^*)^* \cong \g^*
\]
as $\g$-supermodules. Under the isomorphism $\g \cong \g^*$, the identity matrix is sent to the supertrace
\[ \textstyle
\str := \left(\sum_{r=1}^m X_{r,r} \right) - \left(\sum_{r=1}^m X_{r+m,r+m} \right) \in \g^*.
\]
So now we consider the map $\Hbul(\g,k) \rightarrow \Hbul(\g,\g^*)$ that is induced by the homomorphism $k \rightarrow \g^*$ that sends a scalar $\lambda \in k$ to the corresponding scalar multiple of the supertrace. As discussed in Section \ref{subsec:Liebackground}, the cohomology group $\Hbul(\g,\g^*)$ can be computed as the cohomology of the cochain complex $C(\g,\g^*) = \g^* \otimes \Lambda(\g^*)$ with respect to the Koszul differential $\partial$. We will be able to conclude that $f_1 \in J_\g(k^{m|m})$ provided we can show that $\str \otimes (f_1)^p$ is a coboundary in $C^{mp}(\g,\g^*)$.

We view the Koszul differentials $\partial_{\g^*} : C^0(\g,\g^*) \rightarrow C^1(\g,\g^*)$ and $\partial: C^1(\g,k) \rightarrow C^2(\g,k)$ as linear maps $\g^* \rightarrow \g^* \otimes \g^*$ and $\g^* \rightarrow \Lambda_s^2(\g^*)$, respectively. Then in terms of the basis $\set{X_{ij}: 1 \leq i,j \leq 2m}$ for $\g^*$, the Koszul differentials are given by
\begin{align*}
\partial_{\g^*}(X_{ij}) &= \sum_{r=1}^{2m} X_{rj} \otimes X_{ir} - (-1)^{\ol{X_{ir}} \cdot \ol{X_{rj}}} X_{ir} \otimes X_{rj}, \quad \text{and} \\
\partial(X_{ij}) &= \sum_{r=1}^{2m} (-1)^{\ol{X_{ir}} \cdot \ol{X_{rj}}} X_{ir}X_{rj}.
\end{align*}
Now a tedious but straightforward calculation shows that $\partial(f_1) = \str \cdot f_1$; cf.\ \cite[3.7]{Jantzen:1986}. 
This implies that $\partial((f_1)^{p-1}) = (p-1) \cdot \str \cdot (f_1)^{p-1}$. Now set
\[
f_2 = \sum_{i=1}^m \sum_{\sigma \in \fS_m} \sgn(\sigma) X_{\sigma(i),m+i} \otimes \left( \prod_{\ell \neq i} X_{\sigma(\ell),m+\ell} \right) \in \g^* \otimes \Lambda^{m-1}(\g^*) = C^{m-1}(\g,\g^*),
\]
and note that $C(\g,\g^*) = \g^* \otimes \Lambda(\g^*)$ is naturally a right $\Lambda(\g^*)$-supermodule. Then another tedious but straightforward calculation shows that $\partial(f_2) = -\str \otimes f_1 + (-1)^{m-1} f_2 \cdot (\str)$; cf.\ \cite[3.8]{Jantzen:1986}. 
Finally, this implies that $\partial(- f_2 \cdot (f_1)^{p-1}) = \str \otimes (f_1)^p$, which is what we wanted to show.

For the general case of the inclusion $\Chi_\g(M) \subseteq \Chi_\g'(M)$, first suppose that $\g$ is a Lie sub-super\-algebra of $\gl(M)$. Then as discussed in the proof of \cref{nonrestrictedspectrum}, the inclusion $\g \hookrightarrow \gl(M)$ induces a closed embedding $\Chi_\g(k) \hookrightarrow \Chi_{\gl(M)}(k)$, which the reader can check restricts to a closed embedding $\Chi_\g(M) \hookrightarrow \Chi_{\gl(M)}(M)$. From the special case treated in the previous three paragraphs, we know that $\Chi_{\gl(M)}(M) = \Chi_{\gl(M)}'(M)$. This then implies that $\Chi_\g(M) \subseteq \Chi_\g'(M)$.

For the last step of the proof, choose some faithful finite-dimensional $\gzero$-module $P'$ (this is possible by Iwasawa's Theorem), and set $P = U(\g) \otimes_{U(\gzero)} P'$. As a superspace, $P \cong \Lambda(\gone) \otimes P'$. Then $P$ is a faithful finite-dimensional $\g$-module whose superdimension is $0$. By Frobenius reciprocity it follows that $\Ext_\g^\bullet(P,P) \cong \Ext_{\gzero}^\bullet(P',P)$. Since $P'$ and $P$ are finite-dimensional $\gzero$-modules, and since $\gzero$ is a finite-dimensional ordinary Lie algebra, $\Ext_{\gzero}^\bullet(P',P)$ is finite-dimensional. Then $\Ext_\g^\bullet(P,P)$ is finite-dimensional, which implies that the ideal $J_\g(P)$ defining $\Chi_\g(P)$ contains all sufficiently large powers of all non-constant poly\-nomials in $S(\gone^*[p])^{(1)}$. So $\Chi_\g(P) = \set{0}$, hence $\Chi_\g(M) = \Chi_\g(M \oplus P)$ by \cref{standardproperties}(\ref{item:directsum}). Finally, since $M \oplus P$ is a faithful finite-dimensional $\g$-supermodule, $\g$ identifies with a Lie sub-superalgebra of $\gl(M \oplus P)$. Then from the previous paragraph we conclude that $\Chi_\g(M) \subseteq \Chi_\g'(M)$.
\end{proof}

It is noteworthy that the associated varieties introduced by Duflo and Serganova in characteristic zero \cite{Duflo:2005} have precisely the same definition\footnote{The definition given in \cite{Duflo:2005} is in terms of the nonvanishing of a certain quotient vector space, but it can be verified that this is equivalent to the projectivity condition used here.} as the positive characteristic rank variety given in \cref{associatedvariety}. However, unlike here, those varieties have no known cohomological definition. Duflo and Serganova prove that their varieties relate to the combinatorics of atypicality and that their so-called fibre functors provide interesting tensor functors between Lie superalgebras of various ranks.  It would be interesting to adapt their techniques to positive characteristic.

\section{Restricted Lie superalgebras and infinitesimal supergroup schemes}\label{S:restrictedLSA}

As in Section \ref{sec:fdLSAs}, let $k$ be an algebraically closed field of characteristic $p \geq 3$. In this section we present some partial results describing the cohomology variety $\abs{G}$ of an infinitesimal supergroup scheme $G$. We obtain the sharpest results when $G$ is the first Frobenius kernel of the general linear supergroup scheme $GL(m|n)$, or equivalently, when $G$ is the restricted Lie superalgebra $\glmn$. We begin in Section \ref{subsec:CFG} by summarizing some of the main points in the first author's proof \cite{Drupieski:2013,Drupieski:2016} that the cohomology ring $\Hbul(G,k)$ of an infinitesimal supergroup scheme $G$ is a finitely-generated superalgebra. In Section \ref{subsec:resolution} we recall the details of a projective resolution $(X(\g),d_t)$ constructed by Iwai and Shimada \cite{Iwai:1965} and May \cite{May:1966}. In Section \ref{subsec:spectrum} we apply the finite-generation results and explicit calculations using the resolution $(X(\g),d_t)$ to investigate the spectrum of $\Hbul(G,k)$.

\subsection{Cohomological finite generation} \label{subsec:CFG}

Let $\bsp$ be the category of strict polynomial superfunctors as defined in \cite[\S2.1]{Drupieski:2016} (for the original definition, see \cite{Axtell:2013}). In \cite{Drupieski:2016}, the first author calculated for each $r \geq 1$ the structure of $\Ext_{\bsp}^\bullet(\bsir,\bsir)$, the extension algebra in $\bsp$ of the $r$-th Frobenius twist of the identity functor. The functor $\bsir$ admits a direct sum decomposition, $\bsir = \bsirzero \oplus \bsirone$, which gives rise to a matrix ring decomposition
\begin{equation*}
\renewcommand*{\arraystretch}{1.5}
\Ext_{\bsp}^\bullet(\bsir,\bsir) =
\begin{pmatrix}
\Ext_{\bsp}^\bullet(\bsi_0^{(r)},\bsi_0^{(r)}) & \Ext_{\bsp}^\bullet(\bsi_1^{(r)},\bsi_0^{(r)}) \\
\Ext_{\bsp}^\bullet(\bsi_0^{(r)},\bsi_1^{(r)}) & \Ext_{\bsp}^\bullet(\bsi_1^{(r)},\bsi_1^{(r)})
\end{pmatrix}.
\end{equation*}
As an algebra, $\Ext_{\bsp}^\bullet(\bsir,\bsir)$ is then generated by certain distinguished extension classes
\begin{equation} \label{eq:Extgenerators}
\left.
\begin{aligned}
\bse_i' &\in \Ext_{\bsp}^{2p^{i-1}}(\bsi_0^{(r)},\bsi_0^{(r)}) \\
\bse_i'' &\in \Ext_{\bsp}^{2p^{i-1}}(\bsi_1^{(r)},\bsi_1^{(r)})
\end{aligned}
\right\} \text{ for $1 \leq i \leq r$, and }
\left\{
\begin{aligned}
\bsc_r &\in \Ext_{\bsp}^{p^r}(\bsi_1^{(r)},\bsi_0^{(r)}), \\
\bsc_r^\Pi &\in \Ext_{\bsp}^{p^r}(\bsi_0^{(r)},\bsi_1^{(r)}).
\end{aligned}
\right.
\end{equation}
The parity change functor $\Pi$ induces an involutory superalgebra automorphism on $\Ext_{\bsp}^\bullet(\bsir,\bsir)$, denoted $z \mapsto z^\Pi$, that maps $\bsc_r$ to $\bsc_r^\Pi$ and each $\bse_i'$ to $\bse_i''$ (and vice versa).

\begin{theorem} \label{Extalgebrarelations}
The extension algebra $\Ext_{\bsp}^\bullet(\bsir,\bsir)$ is generated by the distinguished extension classes \eqref{eq:Extgenerators} subject only to the relations imposed by the matrix ring decomposition and
	\begin{enumerate}
	\item $(\bse_r')^p = \bsc_r \circ \bsc_r^\Pi$ and $(\bse_r'')^p = \bsc_r^\Pi \circ \bsc_r$.

	\item For each $1 \leq i < r$, $(\bse_i')^p = (\bse_i'')^p = 0$.

	\item For each $1 \leq i < r$, $\bse_i' \circ \bsc_r = \lambda_i (\bsc_r \circ \bse_i'')$ and $\bse_i'' \circ \bsc_r^\Pi = \lambda_i (\bsc_r^\Pi \circ \bse_i')$ for some $\lambda_i \in \set{\pm 1}$. \label{item:lambdas}

	\item The subalgebra generated by $\bse_1',\ldots,\bse_r',\bse_1'',\ldots,\bse_r''$ is commutative.

	\item $\bse_r' \circ \bsc_r = \bsc_r \circ \bse_r''$ and $\bse_r'' \circ \bsc_r^\Pi = \bsc_r^\Pi \circ \bse_r'$. \label{item:ercr}
	\end{enumerate}
\end{theorem}

\begin{proof}
The only relations not covered by \cite[Theorem 4.7.1]{Drupieski:2016} are those in (\ref{item:ercr}). By \cite[Theorem 4.7.1]{Drupieski:2016}, $\bse_r' \circ \bsc_r = \lambda_r (\bsc_r \circ \bse_r'')$ and $\bse_r'' \circ \bsc_r^\Pi = \lambda_r (\bsc_r^\Pi \circ \bse_r')$ for some $\lambda_r \in \set{\pm 1}$. Then $(\bse_r')^p \circ \bsc_r = (\lambda_r)^p \cdot \bsc_r \circ (\bse_r'')^p$. But $(\bse_r')^p = \bsc_r \circ \bsc_r^\Pi$ and $(\bse_r'')^p = \bsc_r^\Pi \circ \bsc_r$, so $(\bsc_r \circ \bsc_r^\Pi) \circ \bsc_r = (\lambda_r)^p \cdot \bsc_r \circ (\bsc_r^\Pi \circ \bsc_r)$. This implies that $(\lambda_r)^p = 1$, and hence that $\lambda_r = 1$.
\end{proof}

\begin{remark}
In \cite[Remark 4.7.2]{Drupieski:2016}, the first author guessed that the $\lambda_i$ should all equal $1$. The previous argument, which the author could have given in \cite{Drupieski:2016} had he had the appropriate presence of mind, shows that at least $\lambda_r = 1$. For a fixed $r$, one can if necessary replace $\bse_i'$ by $-\bse_i'$ in order to assume that $\lambda_i = 1$ for $1 \leq i < r$. But $\bse_i'$ is defined in terms of a distinguished extension class $\bse_i \in \Ext_{\bsp}^{2p^{i-1}}(\bsiizero,\bsiizero)$, so if one is interested in the structure of $\Ext_{\bsp}^\bullet(\bsir,\bsir)$ as $r$ varies, it would be desirable to know that all $\lambda_i$ are equal to $1$. If $\bse_i'$ is replaced by $-\bse_i'$, then it will also no longer be the case that $(\bse_i')^\Pi = \bse_i''$.
\end{remark}

Now let $G$ be an infinitesimal supergroup scheme of height $\leq r$. By \cite[Lemma 4.4.1]{Drupieski:2013}, there exists, for some $m,n \in \N$, a closed embedding $G \hookrightarrow GL(m|n)_r$ of $G$ into the $r$-th Frobenius kernel of the general linear supergroup scheme $GL(m|n)$. Observe that
\begin{equation*} 
\begin{aligned}
\g_m &= \Hom_k(k^{m|0},k^{m|0}), & \g_{+1} &= \Hom_k(k^{0|n},k^{m|0}),\\
\g_n &= \Hom_k(k^{0|n},k^{0|n}), & \g_{-1} &= \Hom_k(k^{m|0},k^{0|n}).
\end{aligned}
\end{equation*}
are each naturally subspaces of $\glmn$, with $\glzero = \g_m \oplus \g_n$ and $\glone = \g_{+1} \oplus \g_{-1}$. Evaluation on the superspace $k^{m|n}$ defines an exact functor from $\bsp$ to the category of rational $GL(m|n)$-supermodules. This functor then induces for each pair $T,T' \in \bsp$ an even linear map
\[
\Ext_{\bsp}^\bullet(T,T') \rightarrow \Ext_{GL(m|n)}^\bullet(T(k^{m|n}),T'(k^{m|n})), \quad z \mapsto z|_{GL(m|n)}
\]
that we call restriction to $GL(m|n)$. Cohomology classes can be restricted further to the sub-supergroup scheme $G$ of $GL(m|n)$; we denote this further restriction by $z|_G$. Then as discussed in \cite[\S5.1]{Drupieski:2016}, the restrictions $\bse_i'|_G$ and $\bse_i''|_G$ for $1 \leq i \leq r$ and $\bsc_r|_G$ and $\bsc_r^\Pi|_G$ define even elements
\begin{equation} \label{eq:linearmaps}
\begin{aligned}
\bse_i'|_G &\in \Ext_G^{2p^{r-1}}(k^{m|0(r)},k^{m|0(r)}) &&\cong \Hom_k(\g_m^{*(r)},\opH^{2p^{r-1}}(G,k)), \\
\bse_i''|_G &\in \Ext_G^{2p^{r-1}}(k^{0|n(r)},k^{0|n(r)}) &&\cong \Hom_k(\g_n^{*(r)},\opH^{2p^{r-1}}(G,k)), \\
\bsc_r|_G &\in \Ext_G^{p^r}(k^{0|n(r)},k^{m|0(r)}) &&\cong \Hom_k(\g_{+1}^{*(r)},\opH^{2p^{r-1}}(G,k)), \text{ and} \\
\bsc_r^\Pi|_G &\in \Ext_G^{p^r}(k^{m|0(r)},k^{0|n(r)}) &&\cong \Hom_k(\g_{-1}^{*(r)},\opH^{2p^{r-1}}(G,k)).
\end{aligned}
\end{equation}
Viewing $(\bse_i'+\bse_i'')|_G$ and $(\bsc_r+\bsc_r^\Pi)|_G$ as linear maps into $\Hbul(G,k)$, they extend uniquely to graded superalgebra homomorphisms
\begin{align*}
(\bse_i'+\bse_i'')|_G &: S(\glzero^*[2p^{i-1}])^{(r)} \rightarrow \Hbul(G,k), \quad \text{and} \\
(\bsc_r+\bsc_r^\Pi)|_G &: S(\glone^*[p^r])^{(r)} \rightarrow \Hbul(G,k).
\end{align*}
Taking the product of these homomorphisms, we get a graded superalgebra homomorphism
\begin{equation} \label{eq:phir}
\phi_G: \left( \bigotimes_{i=1}^r S(\glzero^*[2p^{i-1}])^{(r)} \right) \otimes S(\glone^*[p^r])^{(r)} \rightarrow \Hbul(G,k).
\end{equation}
The main consequence of the results in \cite[\S5.5]{Drupieski:2016} and \cite[\S5.4]{Drupieski:2013} is that $\Hbul(G,k)$ is finite over the image of $\phi_G$ and hence that $\Hbul(G,k)$ is a finitely-generated algebra.

Now consider $GL(m|n)_1$, the first Frobenius kernel of $GL(m|n)$. A key step in the proof that $\Hbul(G,k)$ is finite over $\phi_G$ involved verifying that $(\bse_r'+\bse_r'')|_{GL(m|n)_1}$ and $(\bsc_r+\bsc_r^\Pi)|_{GL(m|n)_1}$ admit particular descriptions. Specifically, consider the May spectral sequence \cite[Corollary 5.2.3]{Drupieski:2013}
\begin{equation} \label{eq:Mayspecseq}
E_0^{i,j} = \Lambda_s^j(\glmn^*) \otimes S^{i/2}(\glzero^*)^{(1)} \Rightarrow \opH^{i+j}(GL(m|n)_1,k).
\end{equation}
Here the superscript $i/2$ means that $E_0^{i,j} = 0$ unless $i$ is even. In the proof of \cite[Theorem 5.5.1]{Drupieski:2016}, the first author verified the following properties:
	\begin{enumerate}[label=(\thesubsection.\arabic*)]
	\setcounter{enumi}{\value{equation}}
	
	\item Replacing $\bse_r'+\bse_r''$ by a scalar multiple if necessary,
	\[
	(\bse_r'+\bse_r'')|_G : S(\glzero^*[2p^{r-1}])^{(r)} \rightarrow \Hbul(GL(m|n)_1,k)
	\]
	is equal to the composition of the $p^{r-1}$-power map $S(\glzero^*[2p^{r-1}])^{(r)} \rightarrow S(\glzero^*[2])^{(1)}$ with the horizontal edge map $E_0^{\bullet,0} \rightarrow \Hbul(GL(m|n)_1,k)$ of \eqref{eq:Mayspecseq}. \label{ertoG1}

	\item Replacing $\bsc_r + \bsc_r^\Pi$ by a scalar multiple if necessary, the composition of
	\[
	(\bsc_r + \bsc_r^\Pi)|_{GL(m|n)_1}: S(\glone^*[p^r])^{(r)} \rightarrow \Hbul(GL(m|n)_1,k)
	\]
	with the vertical edge map $\Hbul(GL(m|n)_1,k) \rightarrow E_0^{0,\bullet}$ of \eqref{eq:Mayspecseq} is equal to the composition of the $p^r$-power map $S(\glone^*[p^r])^{(r)} \rightarrow S(\glone^*)$ and the inclusion into $\Lambda_s(\glmn^*)$. \label{crtoG1}
	
	\setcounter{equation}{\value{enumi}}
	\end{enumerate}
The verification in \cite[\S5.5]{Drupieski:2016} of these properties treated the classes $\bse_r'$, $\bse_r''$, $\bsc_r$, and $\bsc_r^\Pi$ separately, but since $(\bse_r')^\Pi = \bse_r''$, the classes $\bse_r'$ and $\bse_r''$ can be rescaled if necessary by the same scalar factor, and similarly for $\bsc_r$ and $\bsc_r^\Pi$. 

More generally, let $G$ be a sub-supergroup scheme of $GL(m|n)_1$ and set $\g = \Lie(G)$. Then $G = G_1$ and the May spectral sequence for $G$ takes the form
\begin{equation} \label{eq:MayspecseqG}
E_0^{i,j} = \Lambda_s^j(\g^*) \otimes S^{i/2}(\gzero^*)^{(1)} \Rightarrow \opH^{i+j}(G,k).
\end{equation}
Since the May spectral sequence is natural with respect to $G$, the following properties are immediate consequences of \ref{ertoG1} and \ref{crtoG1}:
	\begin{enumerate}[label=(\thesubsection.\arabic*)]
	\setcounter{enumi}{\value{equation}}
	
	\item Up to a scalar factor, the homomorphism $(\bse_r'+\bse_r'')|_G: S(\glzero^*)^{(r)} \rightarrow \Hbul(G,k)$ is equal to the composite map
\[
S(\glzero^*[2p^{r-1}])^{(r)} \rightarrow S(\gzero^*[2p^{r-1}])^{(r)} \rightarrow S(\gzero^*[2])^{(1)} \rightarrow \Hbul(G,k),
\]
where the first arrow is induced by restriction from $\glzero^*$ to $\gzero$, the second arrow is the $p^{r-1}$-power map, and the last arrow is the horizontal edge map of \eqref{eq:MayspecseqG}. \label{ertoG}

	\item Up to a scalar factor, the composition of $(\bsc_r+\bsc_r^\Pi)|_G: S(\glone^*[p^r])^{(r)} \rightarrow \Hbul(G,k)$ with the vertical edge map of \eqref{eq:MayspecseqG} is equal to the composite map
\[
S(\glone^*[p^r])^{(r)} \rightarrow S(\gone^*[p^r])^{(r)} \rightarrow S(\gone^*) \hookrightarrow \Lambda_s(\g^*),
\]
where the first arrow is induced by restriction from $\glone^*$ to $\gone^*$, the second arrow is the $p^r$-power map, and the third arrow is the natural inclusion. \label{crtoG}
	\setcounter{equation}{\value{enumi}}
	\end{enumerate}

\subsection{The projective resolution of Iwai--Shimada and May} \label{subsec:resolution}

Let $\g$ be a finite-dimensional restricted Lie superalgebra over $k$, and write $V(\g)$ for the restricted enveloping superalgebra of $\g$. Iwai and Shimada \cite{Iwai:1965} and May \cite[\S6]{May:1966} described a recipe for constructing a $V(\g)$-free resolution $(X(\g),d_t)$ of the trivial module $k$. We will require some of the details of this construction in order to analyze the maximal ideal spectrum of the cohomology ring $\Hbul(V(\g),k)$. We summarize the necessary details here and refer the reader to \cite[\S3.3]{Drupieski:2013} for more information.

Recall from \cite[\S2.3]{Drupieski:2016} that $\bsa(\g)$ denotes the graded superalgebra of alternating powers on $\g$; it is isomorphic to the graded tensor product of superalgebras $\Lambda(\gzero) \gotimes \Gamma(\gone)$, where $\Gamma(\gone)$ denotes the ordinary divided power algebra on $\gone$. The right adjoint action of $\g$ on itself induces on $\bsa(\g)$ the structure of a right $V(\g)$-supermodule, and hence also of a right $U(\g)$-supermodule. Let $Y(\g) = U(\g) \# \bsa(\g)$ and $W(\g) = V(\g) \# \bsa(\g)$ be the corresponding smash product superalgebras. We consider $Y(\g)$ and $W(\g)$ as homologically graded superspaces with $U(\g)$ and $V(\g)$ concentrated in homological degree $0$ and $\bsa^i(\g)$ concentrated in homological degree $i$. Then $Y(\g)$ identifies with the Koszul resolution for $\g$ as discussed in \cite[\S3.1]{Drupieski:2013}. The Koszul differential on $Y(\g)$ induces an inexact differential $d$ on $W(\g)$, which makes $W(\g)$ into a differential graded superalgebra.

Let $\Gamma(\gzero[2])^{(1)}$ be the divided power algebra on $\gzero$, its vector space structure twisted by the Frobenius map ($\lambda \mapsto \lambda^p$) on $k$, and considered as a homologically graded superspace with $\gzero$ concentrated in $\Z$-degree $2$. The algebra structure on $W(\gzero)$ together with the natural coalgebra structure on $\Gamma(\gzero[2])^{(1)}$ induces on
\[ \textstyle
R = \bigoplus_{n \geq 0} R^n = \bigoplus_{n \geq 0} \bigoplus_{i \geq 0} \Hom_k(\Gamma^i(\gzero[2])^{(1)},W_{i-n}(\gzero))
\]
the structure of a graded superalgebra. Denote the product of elements $r,r' \in R$ by $r \cup r'$. Then an element $t \in R^1$, i.e., a linear map $t: \Gamma(\gzero[2])^{(1)} \rightarrow W(\gzero)$ of homological degree $-1$, is called a \emph{twisting cochain} if $d \circ (t \cup t) = 0$.

Define $X(\g)$ to be the graded superspace $W(\g) \otimes \Gamma(\gzero[2])^{(1)}$. Now given a twisting cochain $t$ as above, the corresponding differential $d_t: X(\g) \rightarrow X(\g)$ is defined as follows: Let $w \in W(\g)$ and $\gamma \in \Gamma(\gzero[2])^{(1)}$ be homogeneous elements. Denote the homological degree of $w$ by $\deg(w)$ and write $\sum \gamma' \otimes \gamma''$ for the coproduct in $\Gamma(\gzero[2])^{(1)}$ of $\gamma$. Then $d_t: X(\g) \rightarrow X(\g)$ is defined by
\begin{equation} \label{eq:dtdifferential} \textstyle
d_t(w \otimes \gamma) = d(w) \otimes \gamma + (-1)^{\deg(w)} \sum [w \cdot t(\gamma')] \otimes \gamma'',
\end{equation}
where $w \cdot t(\gamma')$ denotes the product in $W(\g)$ of $w$ and $t(\gamma') \in W(\gzero) \subseteq W(\g)$.

In \cite{Iwai:1965,May:1966} (see also \cite[Lemma 3.3.1]{Drupieski:2013}) it is shown that a twisting cochain $t$ can always be constructed such that the resulting chain complex $(X(\g),d_t)$ is a $V(\g)$-free resolution of the trivial module. The proof of this fact depends, however, on the choice of a fixed basis for $\gzero$, so the resolution $(X(\g),d_t)$ need not be natural in $\g$. In the construction, the action of $t$ on $\Gamma^i(\gzero[2])^{(1)}$ is defined by induction on $i$ so that the following properties are satisfied:
	\begin{itemize}
	\item[$i=0$:] If $\ve: W(\g) \rightarrow k$ denotes the natural augmentation map on $W(\g)$, then $\ve \circ t = 0$.

	\item[$i=1$:] If $x$ is one of the fixed basis vectors for $\gzero$, then $t(\gamma_1(x)) = x^{p-1} \subgrp{x} - \subgrp{x^{[p]}}$. Here $\gamma_1(x)$ is one of the divided power generators for $\Gamma(\gzero[x])^{(1)}$ (cf.\ \cite[\S2.3]{Drupieski:2013}), $x^{[p]}$ denotes the image of $x$ under the $p$-map making $\gzero$ into a restricted Lie algebra, $x^{p-1}$ is the obvious monomial in $V(\g)$, and $\subgrp{x}$ and $\subgrp{x^{[p]}}$ are the obvious monomials in $\Lambda^1(\gzero) \subset W^1(\g)$.
	\end{itemize}
For the details of the inductive construction, see \cite[Lemma 3.3.1]{Drupieski:2013}.\footnote{More precisely, \cite[Lemma 3.3.1]{Drupieski:2013} asserts the existence of an appropriate linear map $t$ with image in $Y(\gzero)$. To match the exposition presented here, one must then compose this $t$ with the quotient map $Y(\gzero) \rightarrow W(\gzero)$.}

\begin{remark} \label{trivialtwistingcochain}
By \cite[Remark 3.3.2]{Drupieski:2013}, if $\gzero$ is abelian, then $t$ can be constructed to be trivial in homological degrees greater than $2$, i.e., such that $t(\Gamma^i(\gzero[2])^{(1)}) = 0$ for $i > 1$. In the notation used there, if $\gzero$ is abelian, then $r_2 = 0$ regardless of whether or not the $p$-map on $\gzero$ is trivial.
\end{remark}

\begin{remark} \label{extendbasis}
Suppose $\fs$ is a sub-Lie superalgebra of $\g$. Fix a homogeneous basis $S$ for $\fs$, and then extend $S$ to a homogeneous basis $B$ for $\g$. Then the inclusion $\iota: S \hookrightarrow B$ extends to inclusions of chain complexes $Y(\iota): Y(\fs) \hookrightarrow Y(\g)$ and $W(\iota): W(\fs) \hookrightarrow W(\g)$ and an inclusion of graded superspaces $\Gamma(\iota): \Gamma(\fs_{\zero}[2])^{(1)} \hookrightarrow \Gamma(\gzero[2])^{(1)}$. Suppose $t': \Gamma(\fs_{\zero}[2])^{(1)} \rightarrow W(\fs_{\zero})$ is a twisting cochain such that $(X(\fs),d_{t'})$ is a $V(\fs)$-free resolution of the trivial module. Then identifying $\Gamma(\fs_{\zero}[2])^{(1)}$ with its image under $\Gamma(\iota)$, $t'$ can be extended to a twisting cochain $t: \Gamma(\gzero[2])^{(1)} \rightarrow W(\gzero)$ such that $(X(\g),d_t)$ is a $V(\g)$-free resolution of the trivial module. This follows from the argument in \cite[Lemma 3.3.1]{Drupieski:2013} by inductively defining $t$ on the subspace $\Gamma^i(\fs_{\zero}[2])^{(1)}$ of $\Gamma^i(\gzero[2])^{(1)}$ to agree with $t'$ and defining $t$ arbitrarily on any complementary subspace of $\Gamma^i(\gzero[2])^{(1)}$ such that (in the notation of \cite{Drupieski:2013}) $d \circ t_{2i} = r_i$ for $i > 2$. It is possible to define $t$ on $\Gamma^i(\fs_{\zero}[2])^{(1)}$ to agree with $t'$ because the image of $Y(\fs)$ under $Y(\iota)$ is an exact subcomplex of $Y(\g)$. Now since $t$ extends $t'$, the inclusion $\iota$ extends to a monomorphism of projective resolutions $X(\iota): (X(\fs),d_{t'}) \hookrightarrow (X(\g),d_t)$.
\end{remark}

As discussed in \cite{Iwai:1965,May:1966} but not in \cite{Drupieski:2013}, one can construct a diagonal approximation $\Delta_s: X(\g) \rightarrow X(\g) \otimes X(\g)$ in terms of the natural coproducts $\Delta_W$ and $\Delta_\Gamma$ on $W(\g)$ and $\Gamma(\gzero[2])^{(1)}$, the algebra structure of $W(\g)$, and a linear map $s: \Gamma(\gzero[2])^{(1)} \rightarrow W(\gzero) \otimes W(\gzero)$ of homological degree $0$. The map $s$ is called a \emph{twisting diagonal cochain} in \cite{Iwai:1965} and is called a \emph{$t$-twisting coproduct} in \cite{May:1966}. We will not go into the details of the particular properties that $s$ must satisfy, but given the map $s$, and given $w \in W(\g)$ and $\gamma \in \Gamma(\gzero[2])^{(1)}$ as before, $\Delta_s$ is defined by
\begin{equation} \label{eq:Xgcoproduct} \textstyle
\Delta_s(w \otimes \gamma) = \sum [\Delta_W(w) \cdot s(\gamma')] \cdot \Delta_\Gamma(\gamma''),
\end{equation}
where $\Delta_W(w) \cdot s(\gamma')$ denotes the product of $\Delta_W(w) \in W(\g) \otimes W(\g)$ and $s(\gamma') \in W(\gzero) \otimes W(\gzero)$ inside the graded tensor product of superalgebras $W(\g) \gotimes W(\g)$. The supertwist map induces an isomorphism of graded superspaces
\[
(W(\g) \otimes W(\g)) \otimes (\Gamma(\gzero[2])^{(1)} \otimes \Gamma(\gzero[2])^{(1)}) \cong (W(\g) \otimes \Gamma(\gzero[2])^{(1)}) \otimes (W(\g) \otimes \Gamma(\gzero[2])^{(1)}).
\]
Right multiplication by $\Delta_\Gamma(\gamma'')$ in \eqref{eq:Xgcoproduct} then has the evident meaning.

Now $\Hbul(V(\g),k)$ can be computed as the cohomology of the cochain complex $\Hom_{V(\g)}(X(\g),k)$. Applying the duality isomorphisms of \cite[\S2.6]{Drupieski:2016}, there exists an isomorphism of graded superspaces
\begin{equation} \label{eq:HomXg}
\Hom_{V(\g)}(X(\g),k) \cong \Hom_k(\bsa(\g) \otimes \Gamma(\gzero[2])^{(1)},k) \cong \Lambda_s(\g^*) \otimes S(\gzero^*[2])^{(1)}.\footnote{In \cite{Drupieski:2016}, the superexterior algebra $\Lambda_s(\g^*)$ is denoted $\bsl(\g^*)$.}
\end{equation}
The diagonal approximation $\Delta_s$ induces a typically nonassociative product on $\Hom_{V(\g)}(X(\g),k)$. In particular, the induced product on cochains need not agree with the natural superalgebra structure of the tensor product $\Lambda_s(\g^*) \otimes S(\gzero^*[2])^{(1)}$. However, using the fact that the twisting diagonal cochain $s$ has image in the subalgebra $W(\gzero) \otimes W(\gzero)$ of $W(\g) \otimes W(\g)$, one can show that, when restricted to the subspace $S(\gone^*) \otimes S(\gzero^*[2])^{(1)}$ of $\Lambda_s(\g^*) \otimes S(\gzero^*[2])^{(1)}$, the induced product on cochains does agree with the natural superalgebra structure on $S(\gone^*) \otimes S(\gzero^*[2])^{(1)}$.

\begin{remark} \label{Xedgemap}
Let $G$ be the finite $k$-supergroup scheme with $kG = V(\g)$. Then $\Hbul(G,k)$ identifies with $\Hbul(V(\g),k)$, and as discussed in \cite[\S3.5]{Drupieski:2013} the May spectral sequence for $G$ \eqref{eq:MayspecseqG} can be constructed from a filtration on the resolution $X(\g)$. In terms of this construction, the subalgebra $S(\gzero^*[2])^{(1)}$ of $\Hom_{V(\g)}(X(\g),k)$ identifies with the row $j=0$ of \eqref{eq:MayspecseqG}. In particular, $S(\gzero^*[2])^{(1)}$ consists of cocycles in $\Hom_{V(\g)}(X(\g),k)$. Now in terms of the isomorphism \eqref{eq:HomXg}, the horizontal edge map of \eqref{eq:MayspecseqG} is induced by the inclusion of $S(\gzero^*[2])^{(1)}$ into $\Lambda_s(\g^*) \otimes S(\gzero^*[2])^{(1)}$.
\end{remark}

\cref{Xedgemap} enables us to interpret the homomorphism $(\bse_1'+\bse_1'')|_G$ in the $r=1$ case of \ref{ertoG} in terms of the complex $X(\g)$. The next lemma provides an analogous interpretation of the homomorphism $(\bsc_1+\bsc_1^\Pi)|_G$ in the case $r=1$ of \ref{crtoG}.

\begin{lemma} \label{oddfactorization}
Let $m,n \geq 1$, let $\g$ be a finite-dimensional restricted sub-Lie superalgebra of $\glmn$, and let $G$ be the sub-supergroup scheme of $GL(m|n)_1$ with $kG = V(\g)$. Up to a scalar factor, the homomorphism $(\bsc_1 + \bsc_1^\Pi)|_G: S(\glone^*[p])^{(1)} \rightarrow \Hbul(V(\g),k)$ identifies with the composite
\begin{equation} \label{eq:oddcomposite}
S(\glone^*[p])^{(1)} \rightarrow S(\gone^*[p])^{(1)} \rightarrow \Hbul(V(\g),k),
\end{equation}
where the first arrow is restriction from $\glmn$ to $\g$ and the second arrow is induced via \eqref{eq:HomXg} by the $p$-power map $S(\gone^*[p])^{(1)} \rightarrow S(\gone^*)$ and the inclusion $S(\gone^*) \subseteq \Lambda_s(\g^*)$.
\end{lemma}

\begin{proof}
First suppose $\g = \glmn$, so that $G = GL(m|n)_1$, and let $T$ be the subgroup of diagonal matrices in $GL(m|n)$. Then $T$ acts on $\g$ by conjugation, and $(\bsc_1+\bsc_1^\Pi)|_G : S(\glone^*[p])^{(1)} \rightarrow \Hbul(V(\g),k)$ becomes a homomorphism of rational $T$-modules. Let $\Phi_{\odd}$ be the set of weights of $T$ in $\glone^*$ and let $\alpha \in \Phi_{\odd}$. It follows from the argument in the second and third paragraphs of the proof of \cite[Theorem 5.5.1]{Drupieski:2016} that the $p\alpha$-weight space in the $E_0$-page of the May spectral sequence for $G$, i.e., the $p\alpha$-weight space in $\Lambda_s(\g^*) \otimes S(\gzero^*[2])^{(1)}$, is one-dimensional and occurs in total degree $p$. Then the $p\alpha$-weight space in $\Hbul(V(\g),k)$ must also be one-dimensional and concentrated in cohomological degree $p$. Next recall the injective homomorphism $\wt{\varphi}: S(\gone^*[p])^{(1)} \rightarrow \Lambda_s(\g^*)$ discussed in \eqref{eq:ppowermap}, and consider $\Lambda_s(\g^*)$ as a subspace of $\Hom_{V(\g)}(X(\g),k)$ via \eqref{eq:HomXg}. Since the twisting cochain for $X(\g)$ has image in $W(\gzero)$, it follows that the image of $\wt{\varphi}$ consists of cocycles in $\Hom_{V(\g)}(X(\g),k)$. In particular, $\wt{\varphi}(\gone^*[p]^{(1)})$ consists of cocycles in $\Hom_{V(\g)}(X(\g),k)$. Now observe from the construction in \cite[\S3.5]{Drupieski:2013} that the subspace $\Lambda_s(\g^*)$ of $\Hom_{V(\g)}(X(\g),k)$ identifies with the first column of the May spectral sequence for $V(\g)$. Using this, one can deduce that the space $\wt{\varphi}(\gone^*[p]^{(1)})$ consists of cocycle representatives for the subspace of $\Hbul(V(\g),k)$ spanned by all weight vectors of the form $p\alpha$ for $\alpha \in \Phi_{\odd}$.\footnote{Recall that the construction of a twisting cochain making $X(\g)$ into a projective resolution of the trivial module depends on the choice of a fixed basis for $\gzero$. The previous conclusion holds regardless of which basis, hence which twisting cochain, is considered.} Then by dimension comparison, $\wt{\varphi}(\gone^*[p]^{(1)})$ consists of cocycle representatives for the image of $\glone^*[p]^{(1)}$ under $(\bsc_1+\bsc_1^\Pi)|_G$. Finally, the factorization \eqref{eq:oddcomposite} now follows from the $r=1$ case of \ref{crtoG1}.

Now let $\g$ be an arbitrary restricted sub-Lie superalgebra of $\glmn$, and suppose $kG = V(\g)$. As in \cref{extendbasis}, choose homogeneous bases $S$ and $B$ and twisting cochains $t'$ and $t$ for $\g$ and $\glmn$, respectively, such that the inclusion $\iota: S \hookrightarrow B$ extends to a monomorphism of projective resolutions $X(\iota): (X(\g),d_{t'}) \hookrightarrow (X(\glmn),d_t)$. Then the map of cochain complexes
\[
X(\iota)^*: \Hom_{V(\glmn)}(X(\glmn),k) \rightarrow \Hom_{V(\g)}(X(\g),k)
\]
induced by $X(\iota)$ identifies via \eqref{eq:HomXg} with the natural map
\[
\Lambda_s(\glmn^*) \otimes S(\glzero^*[2])^{(1)} \rightarrow \Lambda_s(\g^*) \otimes S(\gzero^*[2])^{(1)}
\]
induced by restriction of linear functions from $\glmn$ to $\g$. Passing to cohomology, $X(\iota)^*$ induces the restriction homomorphism $\Hbul(V(\glmn),k) \rightarrow \Hbul(V(\g),k)$. Since $(\bsc_1+\bsc_1^\Pi)|_G$ factors through the restriction map $\Hbul(V(\glmn),k) \rightarrow \Hbul(V(\g),k)$ by definition, the factorization \eqref{eq:oddcomposite} then follows from the corresponding factorization for $\glmn$.
\end{proof}

\begin{proposition} \label{phifactorsthroughg}
Let $G$ be a sub-supergroup scheme of $GL(m|n)_1$, and set $\g = \Lie(G) \subseteq \glmn$. Then the homomorphism
\begin{equation} \label{eq:phi}
\phi_G: S(\glzero^*[2])^{(1)} \otimes S(\glone^*[p])^{(1)} \rightarrow \Hbul(G,k)
\end{equation}
obtained by taking $r=1$ in \eqref{eq:phir} factors through the restriction homomorphism
\[
S(\glzero^*[2])^{(1)} \otimes S(\glone^*[p])^{(1)} \rightarrow S(\gzero^*[2])^{(1)} \otimes S(\gone^*[p])^{(1)}.
\]
\end{proposition}

\begin{proof}
This is now an immediate consequence of \cref{oddfactorization} and the $r=1$ case of \ref{ertoG}.
\end{proof}

\subsection{Examples} \label{subsec:examples}

\begin{example} \label{smallestrestrictedexample}
Let $\g$ be the restricted Lie superalgebra over $k$ generated by a nonzero odd element $y$ and a nonzero even element $x$ such that $\frac{1}{2}[y,y] = x$ (hence $[x,y]=0$) and $x^{[p]} = x$. Then $\{ x,y \}$ is a homogeneous basis for $\g$. Let $\{ x^*,y^* \}$ be the corresponding dual basis. Then one can use the resolution $X(\g)$ to show that $\Hbul(V(\g),k) \cong k[x^*,y^*]/\subgrp{x^*-(y^*)^2}$, where $x^*$ corresponds to the degree-$2$ polynomial generator of $S(\gzero^*[2])^{(1)}$; for details, see \cite[Example 5.2.1]{Drupieski:2016}.

On the other hand, by \cite[Corollary 5.2.3]{Drupieski:2013} there exists a spectral sequence
\begin{equation} \label{eq:Mayspecseqsmallexample}
E_2^{i,j} = \opH^j(\g,k) \otimes S^{i/2}(\gzero^*)^{(1)} \Rightarrow \opH^{i+j}(V(\g),k),
\end{equation}
By \cref{smallestcalculation}, $\Hbul(\g,k)$ identifies with the truncated polynomial ring $k[y^*]/\subgrp{(y^*)^2}$ generated in cohomological degree $1$. The spectral sequence \eqref{eq:Mayspecseqsmallexample} can also be constructed in terms of a filtration on the resolution $X(\g)$ \cite[\S3.5]{Drupieski:2013}, and using this construction one can then check that the differential $d_2: E_2^{0,1} \rightarrow E_2^{2,0}$ is trivial. This implies by the multiplicative structure of \eqref{eq:Mayspecseqsmallexample} that the spectral sequence halts at the $E_2$-page, and hence that $E_\infty$ is isomorphic as a bigraded algebra to $\Hbul(\g,k) \otimes S(\gzero^*[2])^{(1)} \cong k[x^*,y^*]/\subgrp{(y^*)^2}$, with $x^*$ in bidegree $(2,0)$ and $y^*$ in bidegree $(0,1)$.
\end{example}

\cref{smallestrestrictedexample} shows that there can be algebra relations in the cohomology ring $\Hbul(V(\g),k)$ that cannot be seen by the corresponding May spectral sequence.

\begin{example} \label{needppowers}
Let $\g$ be the restricted Lie superalgebra over $k$ generated by a nonzero odd element $y$ and a nonzero even element $x$ such that $[y,y] = 0$, $[y,x] = y$, and $x^{[p]} = x$. Then $\{x,y \}$ is a basis for $\g$. As in \cite[Example 5.2.1]{Drupieski:2016}, we can write a typical homogeneous monomial in $\bsa(\g) \otimes \Gamma(\gzero[2])^{(1)}$ in the form $\subgrp{x^a} \gamma_b(y)\gamma_c(x)$ for some $a,b,c \in \N$ with $a \leq 1$; cf.\ also the notation in \cite[\S3.3]{Drupieski:2013}. The subalgebra $\gzero$ of $\g$ is abelian, so by \cref{trivialtwistingcochain} we can construct the twisting cochain $t$ to be trivial in cohomological degrees greater than $2$. Then $d_t: X(\g) \rightarrow X(\g)$ satisfies
\begin{align*}
d_t \Big( \subgrp{x^a} \gamma_b(y) \gamma_c(x) \Big) &= d \Big(\subgrp{x^a} \gamma_b(y) \Big) \gamma_c(x) \\
&\relphantom{=} {}+ (-1)^{a+b}\left[(\subgrp{x^a} \gamma_b(y)) \cdot (x^{p-1}\subgrp{x} - \subgrp{x}) \right] \gamma_{c-1}(x) \\
&= x \subgrp{x^{a-1}} \gamma_b(y) \gamma_b(x) \\
&\relphantom{=} {}+ (-1)^a y \subgrp{x^a}\gamma_{b-1}(y)\gamma_c(x) \\
&\relphantom{=} {}+ b \cdot \subgrp{x^{a-1}}\gamma_b(y)\gamma_c(x) \\
&\textstyle \relphantom{=} {}+ (-1)^a\sum_{i=0}^{p-1} \binom{p-1}{i} b^i \cdot x^{p-1-i} \subgrp{x^{a+1}}\gamma_b(y)\gamma_{c-1}(x) \\
&\relphantom{=} {}- (-1)^a \subgrp{x^{a+1}}\gamma_b(y) \gamma_{c-1}(x),
\end{align*}
where $\subgrp{x^{a-1}}$ is interpreted to be zero if $a-1$ is negative. Note that by Fermat's Little Theorem, if $b \not\equiv 0 \mod p$, then $b^{p-1} = 1$ in $k$.

Now let $\{ x^*, y^* \}$ be the basis for $\g^*$ that is dual to $\set{x,y}$. Then a typical monomial in $\Lambda_s(\g^*) \otimes S(\gzero[2]^*)^{(1)}$ can be written in the form $\subgrp{(x^*)^a}(y^*)^b(x^*)^c$ for some $a,b,c \in \N$ with $a \leq 1$. Making the identification \eqref{eq:HomXg}, the differential $d_t^*$ on $\Hom_{V(\g)}(X(\g),k)$ then takes the form
\begin{align*}
(y^*)^b (x^*)^c &\mapsto b \cdot (-1)^b \subgrp{(x^*)}(y^*)^b(x^*)^c, \\
\subgrp{(x^*)}(y^*)^b(x^*)^c &\mapsto (-1)^b (y^*)^b(x^*)^{c+1}, & \text{if $b \equiv 0 \mod p$, and}\\
\subgrp{(x^*)}(y^*)^b (x^*)^c &\mapsto 0 & \text{if $b \not\equiv 0 \mod p$.}
\end{align*}
Then $\Hbul(V(\g),k)$ identifies with the subspace of $\Lambda_s(\g^*) \otimes S(\gzero^*[2])^{(1)}$ spanned by all monomials of the form $(y^*)^{pb}$ for $b \in \N$. In fact, from the comments immediately preceding \cref{smallestrestrictedexample}, we can conclude that $\Hbul(V(\g),k) \cong k[(y^*)^p]$, i.e., $\Hbul(V(\g),k) \cong S(\gone^*[p])^{(1)}$.

A shorter calculation of the algebra structure of $\Hbul(V(\g),k)$ goes as follows. First observe that $\gone$ is an ideal in $\g$, and $\g/\gone$ is isomorphic as a restricted Lie algebra to the toral Lie algebra $\gzero$. Then there exists a Lyndon--Hochschild--Serre spectral sequence
\[
E_2^{i,j} = \opH^i(V(\gzero),\opH^j(V(\gone),k)) \Rightarrow \opH^{i+j}(V(\g),k).
\]
The enveloping superalgebra $V(\gone)$ is isomorphic to the exterior algebra $\Lambda(\gone)$, so $\Hbul(V(\gone),k)$ is isomorphic to the symmetric algebra $S(\gone^*)$ with $\gone^*$ concentrated in cohomological degree $1$. Next, the enveloping algebra $V(\gzero)$ is semisimple, so $E_2^{i,j} = 0$ for all $i > 0$. Then the spectral sequence collapses at the $E_2$-page, yielding an isomorphism of algebras $\Hbul(V(\g),k) \cong S(\gone^*)^{\gzero} = k[y^*]^{\gzero}$. Now since $[x,y] = -y$, $x$ acts on the polynomial generator $y^*$ of $S(\gone^*)$ by $x.(y^*) = y^*$. Finally, since $x$ acts by derivations on $S(\gone^*)$, this implies that $S(\gone^*)^{\gzero}$ is the subalgebra of $S(\gone^*)$ generated by $(y^*)^p$, and hence that $\Hbul(V(\g),k) \cong S(\gone^*[p])^{(1)}$.
\end{example}

The previous example provides some a priori motivation for why the space $\glone^*$ in \eqref{eq:phi} is concentrated in $\Z$-degree $p$.

\begin{example} \label{mainexample}
Let $\g$ be a finite-dimensional restricted Lie superalgebra over $k$ generated by a nonzero odd element $y$ and a nonzero even element $x$ such that $[x,y] = 0$ and $\frac{1}{2}[y,y] = x^{[p]}$. For $i \geq 0$, set $x_i = x^{[p^i]}$. Since $\g$ is finite-dimensional, there exists $n \in \N$ and scalars $\alpha_0,\ldots,\alpha_n \in k$ such that $x_{n+1} = \sum_{i=0}^n \alpha_i x_i$. Assume that $n$ is minimal with this property. Then $\{y,x_0,\ldots,x_n\}$ is a basis for $\g$. With respect to this basis, a homogeneous monomial in $\bsa(\g) \otimes \Gamma(\gzero[2])^{(1)} \cong (\Lambda(\gzero) \gotimes \Gamma(\gone)) \otimes \Gamma(\gzero[2])^{(1)}$ can be written in the form $\subgrp{x_0^{a_0} \cdots x_n^{a_n}} \gamma_b(y) \gamma_{c_0}(x_0) \cdots \gamma_{c_n}(x_n)$ for some $a_i,b,c_i \in \N$ with $a_i \leq 1$ for each $i$. The even subalgebra of $\g$ is abelian, so by \cref{trivialtwistingcochain} we can construct the twisting cochain $t$ to be trivial in cohomological degrees greater than $2$. Then the differential $d_t: X(\g) \rightarrow X(\g)$ satisfies
\begin{align*}
d_t \big( \subgrp{x_0^{a_0} & \cdots x_n^{a_n}} \gamma_b(y) \gamma_{c_0}(x_0) \cdots \gamma_{c_n}(x_n) \big) \\
&= \textstyle \sum_{j=0}^n (-1)^{a_0 + \cdots + a_{j-1}+1} x_j \subgrp{x_0^{a_0} \cdots x_j^{a_j-1} \cdots x_n^{a_n}}\gamma_b(y) \gamma_{c_0}(x_0) \cdots \gamma_{c_n}(x_n) \\
&\relphantom{=} {}+ (-1)^{a_0 + \cdots + a_n} y \subgrp{x_0^{a_0} \cdots x_n^{a_n}} \gamma_{b-1}(y) \gamma_{c_0}(x_0) \cdots \gamma_{c_n}(x_n) \\
&\relphantom{=} {}- \subgrp{x_1x_0^{a_0} \cdots x_n^{a_n}} \gamma_{b-2}(y) \gamma_{c_0}(x_0) \cdots \gamma_{c_n}(x_n) \\
&\relphantom{=} {}+ (-1)^{a_0+\cdots+a_n} \left( \textstyle \sum_{j=0}^n x_j^{p-1}\subgrp{x_0^{a_0} \cdots x_n^{a_n} x_j} \gamma_b(y) \gamma_{c_0}(x_0) \cdots \gamma_{c_j-1}(x_j) \cdots \gamma_{c_n}(x_n) \right) \\
&\relphantom{=} {}- (-1)^{a_0+\cdots+a_n} \left( \textstyle \sum_{j=0}^{n-1} \subgrp{x_0^{a_0} \cdots x_n^{a_n} x_{j+1}} \gamma_b(y) \gamma_{c_0}(x_0) \cdots \gamma_{c_j-1}(x_j) \cdots \gamma_{c_n}(x_n) \right) \\
&\relphantom{=} {}- (-1)^{a_0+\cdots+a_n} \left( \textstyle \sum_{i=0}^n \alpha_i \subgrp{x_0^{a_0} \cdots x_n^{a_n} x_i} \gamma_b(y) \gamma_{c_0}(x_0) \cdots \gamma_{c_{n-1}}(x_{n-1}) \gamma_{c_n-1}(x_n) \right)
\end{align*}

Let $\set{y^*,x_0^*,x_1^*,\ldots,x_n^*}$ be the basis for $\g^*$ that is dual to $\set{y,x_0,x_1,\ldots,x_n}$. Then a homogeneous monomial in $\Lambda_s(\g^*) \otimes S(\gzero[2]^*)^{(1)} \cong (\Lambda(\gzero^*) \gotimes S(\gone^*)) \otimes S(\gzero[2]^*)^{(1)}$ can be written in the form \eqref{eq:dualmonomials}
\begin{equation} \label{eq:dualmonomials}
\subgrp{(x_0^*)^{a_0} \cdots (x_n^*)^{a_n}}(y^*)^b (x_0^*)^{c_0} \cdots (x_n^*)^{c_n},
\end{equation}
for some $a_i,b,c_i \in \N$ with $a_i \leq 1$ for each $i$. Now making the identification of graded superspaces \eqref{eq:HomXg}, the differential $d_t^*$ on $\Hom_{V(\g)}(X(\g),k)$ satisfies
\begin{align*}
(y^*)^b (x_0^*)^{c_0} \cdots (x_n^*)^{c_n} &\mapsto 0,
\\
\subgrp{x_0^*}(y^*)^b (x_0^*)^{c_0} \cdots (x_n^*)^{c_n}
&\mapsto
(-1)^b \alpha_0 \cdot (y^*)^b(x_0^*)^{c_0} \cdots (x_{n-1}^*)^{c_{n-1}}(x_n^*)^{c_n+1},
\\
\subgrp{x_1^*}(y^*)^b (x_0^*)^{c_0} \cdots (x_n^*)^{c_n}
&\mapsto
(-1)^b(y^*)^b(x_0^*)^{c_0+1} (x_1^*)^{c_1} \cdots (x_n^*)^{c_n}
- (-1)^b(y^*)^{b+2}(x_0^*)^{c_0} \cdots (x_n^*)^{c_n}
\\
&\relphantom{\mapsto}
{}+(-1)^b \alpha_1 \cdot (y^*)^b(x_0^*)^{c_0} \cdots (x_n^*)^{c_n+1}, \quad \text{and}
\\
\subgrp{x_i^*}(y^*)^b (x_0^*)^{c_0} \cdots (x_n^*)^{c_n}
&\mapsto
(-1)^b (y^*)^b (x_0^*)^{c_0} \cdots (x_{i-1}^*)^{c_{i-1}+1} \cdots (x_n^*)^{c_n}
\\
&\relphantom{\mapsto} {}+ (-1)^b \alpha_i \cdot (y^*)^b (x_0^*)^{c_0} \cdots (x_{n-1}^*)^{c_{n-1}} (x_n^*)^{c_n+1} \quad \text{if $2 \leq i \leq n$.}
\end{align*}
In particular, the following polynomials are coboundaries:
\begin{align*}
&\alpha_0 \cdot x_n^*, \\
&x_0^* + \alpha_1 \cdot x_n^* - (y^*)^2, \qquad \text{and} \\
&x_i^* + \alpha_{i+1} \cdot x_n^* \qquad \text{for $1 \leq i < n$.}
\end{align*}
More generally, $d_t^*$ maps any monomial of the form \eqref{eq:dualmonomials} in which $\sum_{i=0}^n a_i = j$ to a linear combination of monomials of the form \eqref{eq:dualmonomials} in which $\sum_{i=0}^n a_i = j-1$.
\end{example}

\subsection{The cohomological spectrum} \label{subsec:spectrum}

To begin this section, let $r \geq 1$ and let $GL(m|n)_r$ be the $r$-th Frobenius kernel of the general linear supergroup scheme $GL(m|n)$. As discussed in Section \ref{subsec:CFG}, the extension classes \eqref{eq:Extgenerators} give rise to a homomorphism, which we now write in the form
\begin{equation} \label{eq:simplifiedphir}
\phi_r: S\left( (\glzero^*)^{\oplus r} \oplus \glone^* \right)^{(r)} \rightarrow \Hbul(GL(m|n)_r,k).
\end{equation}
Since $\Hbul(GL(m|n)_r,k)$ is finite over $\phi_r$, the induced morphism between maximal ideal spectra
\begin{equation} \label{eq:Phir}
\Phi_r: \abs{GL(m|n)_r} \rightarrow (\glzero)^{\times r} \times \glone.
\end{equation}
is a finite morphism of affine varieties.

\begin{definition}
Let $G$ be an affine $k$-supergroup scheme and let $\g = \Lie(G)$ be the restricted Lie superalgebra of $G$, with the $p$-map on $\gzero$ denoted by $x \mapsto x^{[p]}$. Given an integer $r \geq 1$, define the commuting variety $C_r(G)$ by
\begin{align*}
C_r(G) = \Big\{ (\alpha_0,\alpha_1,\ldots,\alpha_{r-1},\beta) \in (\gzero)^{\times r} \times \gone &: [\alpha_i,\alpha_j] = 0, [\alpha_i,\beta] = 0 \text{ for all $i,j$,} \\
&\alpha_i^{[p]} = 0 \text{ for $0 \leq i \leq r -2$, and } \alpha_{r-1}^{[p]} = \tfrac{1}{2}[\beta,\beta] \Big\}.
\end{align*}
\end{definition}

Recall that the underlying purely even subgroup scheme of $GL(m|n)$ is $GL_m \times GL_n$. Then the $r$-th Frobenius kernel $(GL_m \times GL_n)_r$ of $GL_m \times GL_n$ is naturally a subgroup scheme of $GL(m|n)_r$, and the inclusion $GL_m \times GL_n \hookrightarrow GL(m|n)$ induces a corresponding morphism
\[
\abs{(GL_m \times GL_n)_r} \rightarrow \abs{GL(m|n)_r}
\]
between the cohomology varieties. We now get the following analogue of \cite[Proposition 5.1]{Suslin:1997}:

\begin{proposition} \label{imageinCrG}
The morphism \eqref{eq:Phir} contains in its image the commuting variety $C_r(GL_m \times GL_n)$. If each $\lambda_i$ in \cref{Extalgebrarelations} is equal to $1$, then the image of \eqref{eq:Phir} is contained in the commuting variety $C_r(GL(m|n))$.
\end{proposition}

\begin{proof}
Set $G = GL(m|n)_r$ and set $A = \Hbul(G,k)$. Write $M(m|n)$ for the affine $k$-superscheme such that for each commutative superalgebra $R$, $M(m|n)(R)$ is equal to the set of all block matrices of the form
\[
g = \left( \begin{array}{c|c} W & X \\ \hline Y & Z \end{array} \right)
\]
with $W$ an $m \times m$ matrix with entries in $R_{\zero}$, $X$ an $m \times n$ matrix with entries in $R_{\one}$, $Y$ an $n \times m$ matrix with entries in $R_{\one}$, and $Z$ an $n \times n$ matrix with entries in $R_{\zero}$. (So $GL(m|n)(R)$ is the principal open subset of $M(m|n)(R)$ defined by the function $\det: g \mapsto \det(W) \cdot \det(Z)$.) Then as in \cite[Remark 3.3]{Suslin:1997}, we can interpret the cohomology classes $\bse_i'|_G$ and $\bse_i''|_G$ for $1 \leq i \leq r$ and $\bsc_r|_G$ and $\bsc_r^\Pi|_G$ as elements of $M(m|n)(R)$, where $R$ is the supercommutative subalgebra
\[
\opH^{\ev}(G,k)_{\zero} \oplus \opH^{\odd}(G,k)_{\one}
\]
of $A$. Specifically, by \eqref{eq:linearmaps} each cohomology class naturally defines a linear map into $\Hbul(G,k)$. Each of these linear maps naturally extends to a linear map $\glmn^{*(r)} \rightarrow \Hbul(G,k)$. For example, the linear map $\g_m^{*(r)} \rightarrow \Hbul(G,k)$ corresponding to $\bse_i'|_G$ extends to $\glmn^{*(r)}$ by acting trivially on the summands $\g_n^{*(r)}$, $\g_{+1}^{*(r)}$, and $\g_{-1}^{*(r)}$ of $\glmn^{*(r)}$. Then the $(i,j)$-entry of the matrix corresponding to a particular cohomology class is equal to the image of the coordinate function $X_{ij}^* \in \glmn^*$ under the corresponding linear map $\glmn^{*(r)} \rightarrow \Hbul(G,k)$.

Interpreting the cohomology classes as elements of $M(m|n)(R)$, the Yoneda product of classes corresponds to matrix multiplication in $M(m|n)(R)$. Then assuming that the structure constants $\lambda_i$ in \cref{Extalgebrarelations} are all equal to $1$, we deduce from \cref{Extalgebrarelations} using reasoning exactly parallel to that in the proof of \cite[Proposition 5.1]{Suslin:1997} that the kernel of \eqref{eq:simplifiedphir} contains a set of generators for the ideal in $S( (\glzero^*)^{\oplus r} \oplus \glone^* )^{(r)}$ defining $C_r(G)$. (Note that if $\beta \in \glone$, then $\frac{1}{2}[\beta,\beta] = \frac{1}{2}(\beta \beta + \beta \beta) = \beta^2$.) This proves the second claim of the proposition.

For the first claim of the proposition, recall from the proof of \cite[Theorem 4.7.1]{Drupieski:2016} that the extension classes $\bse_1',\bse_2',\ldots,\bse_r'$ restrict to scalar multiples of the universal extension classes $e_1^{(r-1)}$, $e_2^{(r-2)}$, \dots, $e_r$ constructed by Friedlander and Suslin \cite{Friedlander:1997}. Somewhat more precisely, restriction to the category $\bsv_{\zero}$ of purely even superspaces defines a functor from the category $\bsp$ of strict polynomial superfunctors to the category $\cp$ of ordinary strict polynomial functors (cf.\ the discussion at the end of \cite[\S2.1]{Drupieski:2016}). This functor then induces a map on extension groups that sends $\bse_i'$ to a scalar multiple of $e_i^{(r-i)}$. Similarly, restriction to the category $\bsv_{\one}$ of purely odd superspaces defines a functor from $\bsp$ to $\cp$ that sends each $\bse_i''$ to a scalar multiple of $e_i^{(r-i)}$. Combining these observations, it follows that the homomorphism \eqref{eq:simplifiedphir} fits into a commutative diagram
\begin{equation} \label{eq:restrictiondiagram}
\vcenter{\xymatrix{
S\left( (\glzero^*)^{\oplus r} \oplus \glone^* \right)^{(r)} \ar@{->}[r]^{\phi_r} \ar@{->}[dd] & \Hbul(GL(m|n)_r,k) \ar@{->}[d] \\
& \Hbul((GL_m \times GL_n)_r,k) \ar@{->}[d] \\
S( (\gl_m^*)^{\oplus r} )^{(r)} \otimes S( (\gl_n^*)^{\oplus r} )^{(r)} \ar@{->}[r] & \Hbul((GL_m)_r,k) \otimes \Hbul((GL_n)_r,k).
}}
\end{equation}
The left-hand vertical arrow in \eqref{eq:restrictiondiagram} is induced by restriction to $\glzero^{\oplus r}$ and by the identification $\glzero = \gl_m \oplus \gl_n$, the right-hand vertical arrows are induced by the inclusion
\[
(GL_m)_r \times (GL_n)_r = (GL_m \times GL_n)_r \hookrightarrow GL(m|n)_r
\]
and by the K\"{u}nneth formula, and the bottom horizontal arrow is the homomorphism induced by Friedlander and Suslin's universal extension classes for $GL_m$ and $GL_n$. Now the first claim of the proposition follows from the commutativity of the diagram and from \cite[\S5]{Suslin:1997a}.
\end{proof}

For $r = 1$ we obtain a sharper result. Set $\Phi = \Phi_1$.

\begin{theorem} \label{T:imageofPhi}
Let $m,n \geq 1$. The image of the morphism $\Phi: \abs{GL(m|n)_1} \rightarrow \glmn$ is precisely
\[
C_1(GL(m|n)) = \set{ (\alpha,\beta) \in \glzero \times \glone : [\alpha,\beta] = 0 \text{ and } \alpha^{[p]} = \tfrac{1}{2}[\beta,\beta]}.
\]
\end{theorem}

\begin{proof}
We have $\im(\Phi) \subseteq C_1(GL(m|n))$ by \cref{imageinCrG}, so let $(\alpha,\beta) \in C_1(GL(m|n))$, and let $\g$ be the restricted subalgebra of $\glmn$ generated by $\alpha$ and $\beta$, i.e., the subalgebra of $\glmn$ generated by $\beta$ and $\alpha^{[p^i]}$ for $i \geq 0$. Let $G$ be the finite supergroup scheme with $kG = V(\g)$. Then $G$ is naturally a sub-supergroup scheme of $GL(m|n)_1$. Now let
\[
\phi_\g: S(\glmn^*)^{(1)} \rightarrow \Hbul(G,k) = \Hbul(V(\g),k)
\]
be the homomorphism obtained by taking $r=1$ in \eqref{eq:phir}, and let $\Phi_\g: \abs{V(\g)} \rightarrow \glmn$ be the corresponding morphism between maximal ideal spectra. By definition, $\phi_\g$ factors through the cohomology ring $\Hbul(GL(m|n)_1,k)$ and $\Phi_\g$ factors through $\abs{GL(m|n)_1}$. We will show that $(\alpha,\beta) \in \im(\Phi)$ by showing that $(\alpha,\beta) \in \im(\Phi_\g)$.

If $\beta = 0$, then $(\alpha,\beta) = (\alpha,0) \in \im(\Phi)$ by \cref{imageinCrG}, so assume that $\beta \neq 0$. If $\alpha = 0$, then $\g$ is the one-dimensional purely odd abelian Lie superalgebra generated by $\beta$, so $V(\g) = U(\g)$ and hence $\Hbul(V(\g),k) = \Hbul(\g,k) \cong S(\g^*)$ by \cref{abelianodd}. Moreover, in this case the May spectral sequence for $V(\g)$ collapses to the column $i = 0$ of the $E_0$-page. Then it follows from the $r=1$ cases of the factorizations \ref{ertoG} and \ref{crtoG} that $(0,\beta) \in \im(\Phi_\g)$. So now assume that $\alpha$ and $\beta$ are both nonzero. Then $\g$ is of the type considered in \cref{mainexample}. Using the factorizations in \ref{ertoG}, \cref{Xedgemap}, and \cref{oddfactorization}, it follows from the explicit calculations in \cref{mainexample} that $(\alpha,\beta) \in \im(\Phi_\g)$.
\end{proof}

More generally, let $G$ be a sub-supergroup scheme of $GL(m|n)_1$, and let $\g = \Lie(G)$, considered as a sub-Lie superalgebra of $\glmn$.  By \cref{phifactorsthroughg}, the homomorphism
\[
\phi_G: S(\glmn^*)^{(1)} \rightarrow \Hbul(G,k)
\]
factors through $S(\g^*)$, so the corresponding morphism of affine varieties
\[
\Phi_G: \abs{G} \rightarrow \glmn
\]
has image in the subspace $\g$ of $\glmn$. Then arguing as in the proof of \cref{T:imageofPhi}, one gets:

\begin{proposition}
$\im(\Phi_G) = C_1(G)$.
\end{proposition}

The morphisms $\phi_G$ and $\Phi_G$ depend intrinsically on the particular embedding of $G$ into $GL(m|n)_1$, and a different choice of embedding could result in a different finite morphism. It would be desirable to have a description of $\abs{G}$ that is independent of the choice of embedding $G \hookrightarrow GL(m|n)_1$.

\makeatletter
\renewcommand*{\@biblabel}[1]{\hfill#1.}
\makeatother

\bibliographystyle{eprintamsplain}
\bibliography{on_support_varieties}

\end{document}